\documentclass[11pt]{amsart}
\usepackage[T1]{fontenc}
\usepackage{lmodern,microtype}
\usepackage{amsmath,amssymb,amsthm,mathtools}
\usepackage{enumitem}
\usepackage[margin=1.05in]{geometry}
\usepackage[colorlinks=true,linkcolor=blue!50!black,citecolor=blue!50!black,urlcolor=blue!50!black]{hyperref}
\usepackage{xcolor}
\usepackage[capitalize]{cleveref}

\newtheorem{theorem}{Theorem}[section]
\newtheorem{proposition}[theorem]{Proposition}
\newtheorem{lemma}[theorem]{Lemma}
\newtheorem{corollary}[theorem]{Corollary}
\theoremstyle{remark}
\newtheorem{remark}[theorem]{Remark}
\newtheorem{example}[theorem]{Example}

\newcommand{\R}{\mathcal R}
\newcommand{\RK}{\mathcal R_K}
\newcommand{\res}{\operatorname{res}}
\newcommand{\diag}{\operatorname{diag}}
\newcommand{\GL}{\operatorname{GL}}
\newcommand{\PGL}{\operatorname{PGL}}
\newcommand{\Sch}[2]{\{#1\}_{#2}}

\newcommand{\I}{I_2}
\numberwithin{equation}{section}

\title[Annular Frobenius for Schwarzian equations]{Annular Frobenius Classification of\texorpdfstring{\\}{ }$p$-adic Stieltjes--Schwarzian Equations}
\author{Mohammadreza Mohajer}
\address{Department of Mathematics, Physics and Geology, Cape Breton University, 1250 Grand Lake Road, Sydney, Nova Scotia B1P 6L2, Canada}
\email{Mohammadreza\_Mohajer@cbu.ca, mohammadreza.mohajer96@gmail.com}
\author{Abdellah Sebbar}
\address{Department of Mathematics and Statistics, University of Ottawa, Ottawa, Ontario K1N 6N5, Canada}
\email{asebbar@uottawa.ca}

\date{}

\hypersetup{
  pdftitle={Annular Frobenius Classification of p-adic Stieltjes--Schwarzian Equations},
  pdfauthor={Mohammadreza Mohajer and Abdellah Sebbar},
  pdfsubject={Annular normal forms and Frobenius structures for p-adic Schwarzian differential equations},
  pdfkeywords={p-adic differential equations, Schwarzian equations, Robba ring, Frobenius structures, Heine--Stieltjes polynomials, Jacobi polynomials, modular differential equations}
}

\begin{document}

\begin{abstract}
We give explicit annular normal forms and classify Frobenius
structures for $p$-adic Schwarzian equations defined by developing
differentials. For a unit of the Robba ring, reduction of order
expresses the associated module as a quadratic rank-one twist of a
unipotent extension determined by the residue. Frobenius existence
is characterized by a square-class condition. At odd primes, there
are exactly four differential-module types, and a uniform
dominant-monomial hypothesis gives a Frobenius formula in analytic
families across the zero-residue locus. For rational nonintegral
annular powers, we obtain a diagonal normal form and a congruence
criterion at odd primes; both results hold at every prime when the
unit factor is an explicit square. Applied to the reducible modular
equations $\theta^2y-\ell^2E_4y/144=0$, with $\gcd(\ell,6)=1$,
this gives exactly two annular types at every prime and determines
the least $s\geq1$ for which Frobenius under $q\mapsto q^{p^s}$
exists. The subfamily $\ell=12n+1$ has a single annular type.
Heine--Stieltjes residue cancellation and nondegenerate Jacobi
configurations provide explicit examples. We also prove an
obstruction to constant projective Frobenius equivariance of
Robba-ring developing maps.
\end{abstract}

\maketitle
\section{Introduction}
A developing differential determines a Schwarzian
coefficient and an associated second-order differential equation.
Classical Heine--Stieltjes theory supplies such differentials
through residue cancellation at the roots of a polynomial
\cite{Stieltjes,Scherbak}. In the modular setting, integrals of
meromorphic forms of weight two give developing maps for modular
Schwarzian equations \cite[\S7]{SSintegrals}. These constructions
lead to a classification problem for the associated differential
modules on a $p$-adic annulus. We construct explicit normal forms
and all Frobenius matrices, including a formula valid in relative
analytic families that remains valid when the residue vanishes. For the classical reducible modular family, we establish convergence of the gauges at every prime and determine the annular differential-module types. The least Frobenius iterate then follows from the exponent classes.

Let $K/\mathbb Q_p$ be finite, and let $\RK$ denote the Robba ring
over $K$ in the variable $q$, consisting of Laurent series
converging on some open annulus $\rho<|q|<1$. Put
$\theta=q\,d/dq$. For $d=p^s$, with $s\ge1$, let $\varphi_d$ act
by $q\mapsto q^d$ and fix $K$. Throughout, a Frobenius structure
means a horizontal isomorphism $\varphi_d^*M\to M$.
Section~\ref{sec:prelim} specifies the differential-module and
pullback conventions, following \cite{Kedlaya,KedOverview}.
For $g\in\RK^\times$, define
\begin{equation}\label{eq:Fg}
u=\frac{\theta g}{g},\qquad
F_g=\theta u-\frac12u^2,\qquad
M(g):\quad \theta^2y+\frac12F_gy=0.
\end{equation}
When $\theta h=g$, the coefficient $F_g$ is the Schwarzian
derivative of $h$. The logarithmic-derivative formula defines
$M(g)$ for every unit $g$, including those for which the
developing differential $g\,dq/q$ has nonzero residue.

Our first result expresses $M(g)$ in terms of a rank-one factor
and an annular extension class. Write
$g=r+\theta H$, where $r=\res(g\,dq/q)$ is the constant Laurent
coefficient of $g$ and $H\in\RK$ has zero constant coefficient.
An explicit reduction-of-order gauge in \cref{thm:normal}
identifies $M(g)$ with
\begin{equation}\label{eq:intro-normal}
\theta Z=\left(-\frac{\theta g}{2g}I_2+rN\right)Z,
\qquad
N=\begin{pmatrix}0&1\\0&0\end{pmatrix}.
\end{equation}
The scalar term defines a rank-one module whose tensor square
is trivial, while the unipotent factor has extension class
$r[dq/q]\in H^1_{\mathrm{dR}}(\RK)$.
Thus the residue determines whether the unipotent extension splits.
Building on this normal form, \cref{thm:frobenius} describes every
$\varphi_d$-Frobenius structure and proves that one exists exactly
when
\begin{equation}\label{eq:intro-square}
\frac{\varphi_d(g)}{g}\in K^\times(\RK^\times)^2.
\end{equation}
For odd $p$, this condition holds for every unit $g$.
Moreover, \cref{cor:four} gives exactly four differential-module
isomorphism types, determined by the parity of the annular winding
number and the vanishing or nonvanishing of $r$.

The residue normal form extends to analytic families.
For odd $p$, under a uniform dominant-monomial hypothesis,
\cref{thm:family} constructs a single explicit Frobenius matrix
over the relative Robba ring. The formula remains defined across
the zero-residue locus, where the unipotent extension changes
from nonsplit to split. In particular, this change of
differential-module type is compatible with an analytic
Frobenius structure on the entire family.
\Cref{thm:stieltjes-family} applies the construction to local
analytic branches of nondegenerate Stieltjes configurations
whose annular pullbacks satisfy the uniform hypothesis.
Explicit examples exhibit moving Stieltjes roots and a varying
residue without degeneration of the root configuration.

Developing differentials with fractional powers lead to a
further arithmetic condition. For
$\nu\in\mathbb Q\setminus\mathbb Z$ and $f\in\RK^\times$, the
expression $q^\nu f$ determines a coefficient in $\RK$ through
$u=\nu+\theta f/f$ and $F=\theta u-u^2/2$.
For odd $p$, \cref{thm:kummer} gives a diagonal normal form over
$\RK$ and proves that Frobenius exists precisely when
\begin{equation}\label{eq:intro-kummer}
\frac{(d-1)\nu}{2}\in\mathbb Z
\quad\text{or}\quad
\frac{(d+1)\nu}{2}\in\mathbb Z.
\end{equation}
The two alternatives correspond to Frobenius preserving or
interchanging the two rank-one summands. When $f$ is an explicit
square in $\RK^\times$, \cref{cor:kummer-square} gives the same
construction and criterion at every prime. This case provides the
annular normal forms needed for the modular application.

The principal application concerns the classical
Hurwitz--Klein family
\begin{equation}\label{eq:intro-modular}
\theta^2y-\frac{(12n+1)^2}{144}E_4(q)y=0,
\qquad n\ge0,
\end{equation}
where
\[
E_4(q)=1+240\sum_{k\ge1}\sigma_3(k)q^k,
\qquad
\sigma_3(k)=\sum_{\substack{a\mid k\\a>0}}a^3.
\]
At every prime, \cref{thm:modular-family} identifies all these
modules over $\mathcal R_{\mathbb Q_p}$ with the diagonal
system having exponents $-1/12$ and $1/12$. The identification
is given by an explicit gauge whose convergence is established
on a Robba annulus. Consequently, Frobenius exists exactly when
$d\equiv\pm1\pmod{12}$. For $p\ge5$, the least integer $s\ge1$
for which $\varphi_{p^s}$-Frobenius exists is one when
$p\equiv1,11\pmod{12}$ and two when
$p\equiv5,7\pmod{12}$. No positive $s$ exists at $p=2,3$.

\Cref{cor:all-modular} extends the classification to
\[
\theta^2y-\frac{\ell^2}{144}E_4(q)y=0,
\qquad \ell\ge1,\quad \gcd(\ell,6)=1.
\]
This full reducible modular family has exactly two annular
differential-module types, corresponding to
$\ell\equiv\pm1$ and $\ell\equiv\pm5\pmod{12}$.
Both types have the same Frobenius existence criterion.
For $p\equiv5,7\pmod{12}$, pullback under $q\mapsto q^p$
exchanges the two types, explaining why two iterations are
necessary and sufficient.

The Stieltjes and modular constructions draw on established
algebraic results. The polynomial root equations and their
master-function interpretation belong to Heine--Stieltjes
theory \cite{Stieltjes,Scherbak}, and the Wronskian pencils arise
from the reproduction procedure of Mukhin and Varchenko
\cite{MV}. Korotkin develops the relation between
Stieltjes--Bethe equations and projective structures
\cite{Korotkin}. Saber and Sebbar construct the complex modular
developing differentials underlying \eqref{eq:intro-modular}
\cite{SSintegrals}. The developing differentials for the full
reducible family and their Jacobi description appear in
\cite[Theorems~3.1 and~6.1]{BSS}. Our modular application adds
the annular identification, with convergent gauges over
$\mathcal R_{\mathbb Q_p}$ at every prime, and the resulting
Frobenius classification. The arithmetic Frobenius action in
\cite{BSS} concerns pole fields; here Frobenius is a horizontal
isomorphism under $q\mapsto q^d$.

The analytic ingredients include annular integration,
rank-one Kummer equations, and the classification of unipotent
connections \cite[\S\S7.3,8.2,9.2,17.1]{Kedlaya};
the nilpotent linear-algebra description is given in
\cite[Proposition~4.5.3]{KedOverview}. General results on
$p$-adic differential equations in families appear in
\cite{KedFamilies}. The four-type classification follows from
these standard ingredients once the Schwarzian gauge is explicit;
the family formula remains defined as the extension splits.
The Jacobi Hessian calculation supplies a concrete
nondegeneracy criterion for the deformation construction,
using the spectral interpretation in \cite{Sasaki}.
Finally, these results give the annular classification of the
reducible modular family. A developing map is obtained as the ratio
of two horizontal solutions, but \cref{thm:rigidity} shows that a Robba-ring developing map with unit derivative cannot satisfy a constant projective Frobenius functional equation, even when its differential module admits Frobenius.

Sections~\ref{sec:annular} and \ref{sec:kummer} establish the
normal forms and Frobenius classifications for unit and Kummer
developing differentials, respectively.
Section~\ref{sec:stieltjes} recalls Stieltjes residue cancellation,
its projective interpretation, and the Wronskian degeneracy
mechanism, with a brief discussion of global restrictions.
Section~\ref{sec:families} develops the analytic family
construction and its Stieltjes and Jacobi examples.
Section~\ref{sec:modular} proves the modular applications.

\section{Schwarzian equations and annular preliminaries}\label{sec:prelim}

\subsection{Developing maps and changes of coordinates}

Let $(E,D)$ be a differential field of characteristic zero, with constant field $C=\ker D$. For $h\in E$ with $Dh\ne0$, define the Schwarzian derivative \cite[\S2]{SSintegrals} by
\begin{equation}\label{eq:schwarzian}
 \Sch{h}{D}=D\left(\frac{D^2h}{Dh}\right)
 -\frac12\left(\frac{D^2h}{Dh}\right)^2
 =\frac{D^3h}{Dh}-\frac32\left(\frac{D^2h}{Dh}\right)^2.
\end{equation}
A solution of $\Sch{h}{D}=F$ is called a \emph{developing map} for the coefficient $F$. In an analytic coordinate, local univalence means that the derivative does not vanish. In the annular statements below we require the stronger ring-theoretic condition $\theta h\in\RK^\times$.

The next proposition records the classical Schwarzian correspondence and chain rule; see \cite[\S2]{SSintegrals} and \cite[\S1]{Korotkin}. We give the algebraic proof to specify the constants and square roots involved.

\begin{proposition}[Classical Schwarzian identities]\label[proposition]{prop:formalism}
Suppose $y_1,y_2$ solve $D^2y+Fy/2=0$ in a differential field extension and have nonzero Wronskian. Then $h=y_2/y_1$ satisfies $\Sch{h}{D}=F$. Conversely, if $Dh=g\ne0$, adjoining $g^{1/2}$ gives the fundamental pair
\[
 y_1=g^{-1/2},\qquad y_2=hg^{-1/2}
\]
for the equation with coefficient $\Sch{h}{D}$.
For $x\in E$ with $Dx\ne0$ and $D_x=(Dx)^{-1}D$, one has
\begin{equation}\label{eq:chain}
 \Sch{h}{D}=(Dx)^2\Sch{h}{D_x}+\Sch{x}{D}.
\end{equation}
In particular, postcomposition by an element of $\PGL_2(C)$ preserves the Schwarzian, and replacing $D$ by $cD$, $c\in C^\times$, multiplies it by $c^2$.
\end{proposition}

\begin{proof}
The Wronskian $W=y_1Dy_2-y_2Dy_1$ satisfies $DW=0$. For $h=y_2/y_1$ we have $Dh=W/y_1^2$, whence
\[
 \Sch{h}{D}=-2D\left(\frac{Dy_1}{y_1}\right)
             -2\left(\frac{Dy_1}{y_1}\right)^2=F.
\]
In the converse direction, $Dy_1/y_1=-Dg/(2g)$ gives the equation for $y_1$. The equation for $hy_1$ follows by differentiation, and the Wronskian is $y_1^2Dh=1$. Linear independence is over the constants of the extension in which these solutions lie.

To prove \eqref{eq:chain}, put $b=Dx$, $s=Db/b$, and $v=D_x^2h/D_xh$. Then $D^2h/Dh=s+bv$. Upon substituting in \eqref{eq:schwarzian}, the mixed terms cancel because $Db=sb$. What remains is $Ds-s^2/2+b^2(D_xv-v^2/2)$. For a fractional linear map $T$, direct differentiation gives $\Sch{T}{d/dx}=0$, so the chain rule proves projective invariance. Constant rescaling follows from the definition.
\end{proof}

A \emph{projective connection} on a curve is described in local coordinates by coefficients related by \eqref{eq:chain}; see \cite[\S1]{Korotkin}. Thus the difference of two projective connections is a quadratic differential. This terminology does not identify a scalar coefficient with an arbitrary rank-two connection: a scalar equation includes a chosen cyclic presentation. An unrestricted change of basis of the connection may change that presentation.

The logarithmic-derivative reduction of the scalar equation \cite[\S2]{SSintegrals} gives the Riccati identity
\begin{equation}\label{eq:riccati}
 a=-\frac{Dg}{2g},\qquad Da+a^2=-\frac12F_g.
\end{equation}
It fixes the sign in the factorization $D^2+F_g/2=(D+a)(D-a)$, where operators act on the right-hand argument and $Da$ in \eqref{eq:riccati} denotes differentiation of the coefficient.

\subsection{The Robba ring, integration, and winding}

From now on, $K$ is a finite extension of $\mathbb Q_p$, normalized by $|p|=p^{-1}$, and $q$ is an indeterminate. First define $\theta=q\,d/dq$ on Laurent series. The \emph{Robba ring} $\RK$ consists of the series $f=\sum_{n\in\mathbb Z}a_nq^n$ for which, for some $0<\rho<1$,
\[
 |a_n|t^n\longrightarrow0\quad(n\longrightarrow\pm\infty)
 \qquad\text{for every }\rho<t<1.
\]
It is a differential integral domain, with constants $\ker\theta=K$. On such an annulus the multiplicative Gauss norm is $|f|_t=\max_n|a_n|t^n$. The ring is the union over inner radii, with outer radius always $1$; it is not the ring of functions on one closed annulus. These conventions agree with \cite[Definition~15.1.4 and \S8.4]{Kedlaya}.

The residue of $f\,dq/q$ is, by definition, its coefficient $a_0$. This is the residue used in annular de Rham cohomology \cite[\S4.4]{KedOverview}; it does not involve integration around a topological circle. The following standard annular integration calculation is \cite[Lemma~9.2.1]{Kedlaya}, written for the logarithmic derivation.

\begin{lemma}[Residue and integration]\label[lemma]{lem:integration}
Every $g=\sum b_nq^n\in\RK$ has a unique expression
\[
 g=r+\theta H,\qquad r=b_0,\qquad
 H=\sum_{n\ne0}\frac{b_n}{n}q^n,
\]
with zero constant coefficient in $H$. In particular,
\[
 H^1_{\mathrm{dR}}(\RK):=
 \frac{\RK\,dq/q}{\{\theta f\,dq/q:f\in\RK\}}\simeq K
\]
by residue. If $f\ne0$ and $\theta f=\lambda f$ for $\lambda\in K$, then $\lambda\in\mathbb Z$ and $f=cq^\lambda$ with $c\in K^\times$.
\end{lemma}

\begin{proof}
The formal statements follow by comparing coefficients. For convergence, fix $\rho<t<1$ and choose $\rho<t_-<t<t_+<1$. For positive $n$, the terms of $H$ are bounded using convergence of $g$ at $t_+$ and the factor $|1/n|_p(t/t_+)^n$; for negative $n$, use $t_-$ and $|1/n|_p(t_-/t)^{|n|}$. Since $|1/n|_p\le |n|_{\infty}$, exponential decay dominates this polynomial factor. Thus $H$ converges at every radius in the same open annulus. The last assertion is the coefficient identity $(n-\lambda)a_n=0$.
\end{proof}

We will use the annular unit criterion of \cite[Lemma~8.2.6]{Kedlaya}. The integer in the next lemma is called the \emph{winding number} here; the residue formula makes its meaning intrinsic to the chosen oriented annular parameter.

\begin{lemma}[Units and winding number]\label[lemma]{lem:units}
Every $g\in\RK^\times$ has, on a sufficiently small annular germ, a factorization
\begin{equation}\label{eq:dominant}
 g=cq^m(1+e),\qquad c\in K^\times,\quad m\in\mathbb Z,
 \quad |e|_t<1\quad(\rho<t<1).
\end{equation}
The integer $m$ is unique and equals
\begin{equation}\label{eq:winding}
 m(g)=\res\left(\frac{dg}{g}\right)
      =\res\left(\frac{\theta g}{g}\frac{dq}{q}\right).
\end{equation}
It is additive under products and satisfies $m(g(q^e))=e\,m(g)$ for every positive integer $e$. For odd $p$, one can write $g=cq^m v^2$ with $v\in\RK^\times$.
\end{lemma}

\begin{proof}
Choose an annulus on which both $g$ and $g^{-1}$ converge. On every closed subannulus, \cite[Lemma~8.2.6]{Kedlaya} gives a unique dominant Laurent monomial. Its exponent agrees on overlapping subannuli, so it is one integer $m$ throughout the open annulus. Taking $c$ to be the coefficient of that monomial gives \eqref{eq:dominant}.

The logarithm $\log(1+e)=\sum_{k\ge1}(-1)^{k+1}e^k/k$ converges on every closed subannulus: the endpoint norms of $e$ are strictly less than one. Its derivative is $\theta e/(1+e)$. The residue of that derivative is zero by \cref{lem:integration}, proving \eqref{eq:winding}. Additivity follows from logarithmic differentiation, and substitution proves the pullback formula. For odd $p$, the binomial coefficients of $(1+e)^{1/2}$ belong to $\mathbb Z_p$; this series and its inverse converge on the same subannuli. Their compatible sums give $v$.
\end{proof}

The two residues $\res(g\,dq/q)$ and $\res(dg/g)$ play different roles. The former is an arbitrary element of $K$ measuring exactness of the developing differential. The latter is an integer measuring the monomial factor of a unit. Neither residue is assumed to determine the other.

\subsection{Differential modules, gauges, and Frobenius pullback}

A differential module over $(\RK,\theta)$ is a finite free module $M$ with a $K$-linear operator $\nabla_\theta$ satisfying
\[
 \nabla_\theta(fm)=\theta(f)m+f\nabla_\theta(m).
\]
In a basis, write $\nabla_\theta=\theta-B$. Its horizontal sections satisfy $\theta Y=BY$. A \emph{gauge} is an invertible matrix $P$ implementing a basis change $Y=PZ$; the new horizontal matrix is
\begin{equation}\label{eq:gauge}
 A=P^{-1}BP-P^{-1}\theta P,
 \quad\text{equivalently}\quad \theta P=BP-PA.
\end{equation}
These are the standard differential-module conventions of \cite[\S\S5.1--5.3]{Kedlaya}, expressed in terms of horizontal solution matrices.

Fix $d=p^s$ and write $\varphi=\varphi_d$ for $f(q)\mapsto f(q^d)$, with $K$ fixed. It maps series on $\rho<|q|<1$ to series on $\rho^{1/d}<|q|<1$, hence defines an endomorphism of $\RK$, and $\theta\varphi=d\varphi\theta$. In the terminology of \cite[Definition~15.2.1]{Kedlaya}, this is a relative $d$-power Frobenius lift. If the residue field of $K$ has $p^f$ elements, it is an absolute lift exactly when $f$ divides $s$.

The connection pullback \cite[\S5.3 and Remark~17.1.2]{Kedlaya} is the tensor product
\[
 \varphi^*M=\RK\otimes_{\varphi,\RK}M,
 \qquad 1\otimes fm=\varphi(f)\otimes m,
\]
with connection
\[
 \nabla_\theta^*(a\otimes m)=\theta(a)\otimes m
                  +d a\otimes\nabla_\theta m.
\]
The identity $\theta\varphi=d\varphi\theta$ makes this well defined. Its horizontal matrix in the pulled-back basis is $d\varphi(B)$.

A \emph{Frobenius structure} is a horizontal isomorphism $\varphi^*M\to M$; see \cite[Definition~17.1.1 and Remark~17.1.2]{Kedlaya}. Its matrix $\Phi\in\GL_r(\RK)$ satisfies
\begin{equation}\label{eq:horizontal}
 \theta\Phi=B\Phi-\Phi\,d\varphi(B).
\end{equation}
Thus Frobenius is extra structure on a differential module, and not an equality between its two displayed matrices.

For a Schwarzian coefficient $F$, the companion matrix is
\begin{equation}\label{eq:companion}
 B_F=\begin{pmatrix}0&1\\-F/2&0\end{pmatrix},
 \qquad Y=\binom{y}{\theta y}.
\end{equation}
Let $D_d=\diag(1,d)$. Direct multiplication gives
\[
 D_d\,d\varphi(B_F)D_d^{-1}=B_{d^2\varphi(F)}.
\]
Consequently, the normalized matrix $G=\Phi D_d^{-1}$ satisfies
\begin{equation}\label{eq:normalized-gauge}
 \theta G=B_FG-GB_{d^2\varphi(F)}.
\end{equation}
Taking traces shows that $\det\Phi,\det G\in K^\times$, and $\det\Phi=d\det G$. We retain this distinction whenever a determinant is specified.

\section{Unit developing differentials}\label{sec:annular}

\subsection{Reduction of order and the residue class}
A rank-one module is \emph{trivial} if it has an invertible
horizontal section, in the usual differential-module sense
\cite[\S5.1]{Kedlaya}. Tensoring by a rank-one module is called a
\emph{rank-one twist}; horizontal matrices acquire the corresponding
scalar term. A module is \emph{unipotent} if it has a filtration with
trivial rank-one quotients
\cite[Definition~4.5.1]{KedOverview}. The standard classification of
unipotent annular connections uses constant nilpotent matrices
\cite[Proposition~4.5.3]{KedOverview}.

Set $N=\begin{pmatrix}0&1\\0&0\end{pmatrix}$. For $b\in\RK$, the unipotent horizontal system
\[
U_b:\qquad \theta Z=bNZ
\]
is an extension of the trivial rank-one module by itself. We use the
horizontal-matrix convention, under which its extension class is
\[
[b\,dq/q]\in H^1_{\mathrm{dR}}(\RK).
\]
For the Schwarzian construction, the following calculation derives
both the rank-one twist and the unipotent extension directly from the
developing differential.

\begin{theorem}[Residue normal form]\label[theorem]{thm:normal}
Let $g\in\RK^\times$, let $g=r+\theta H$ as above, and put $a=-\theta g/(2g)$. The matrix
\begin{equation}\label{eq:P}
 P_g=\begin{pmatrix}1&H\\a&g+aH\end{pmatrix},
 \qquad \det P_g=g,
\end{equation}
transforms $M(g)$ into
\begin{equation}\label{eq:normal}
 \theta Z=(a\I+rN)Z.
\end{equation}
If $L_g$ denotes the rank-one system $\theta z=az$ and $U_r$ denotes $\theta Z=rNZ$, then
$$
 M(g)\simeq L_g\otimes U_r.
$$
The tensor square of $L_g$ is trivial. The extension $U_r$ splits if and only if $r=0$; for $r\ne0$, it is isomorphic to $U_1$ and is indecomposable.
\end{theorem}

\begin{proof}
The identities $\theta g=-2ag$ and $-F_g/2=\theta a+a^2$ give, by multiplication,
$$
 \theta P_g=B_{F_g}P_g-P_g(a\I+rN).
$$
Since $g$ is a unit, $P_g$ is an invertible gauge. The system for $L_g^{\otimes2}$ has the invertible solution $g^{-1}$, so this tensor square is trivial.

The splitting calculation is the rank-two case of \cite[Proposition~4.5.3]{KedOverview}. The standard subobject and quotient of $U_r$ are trivial rank-one systems. A splitting is equivalent to solving $\theta f=r$, which by \cref{lem:integration} is possible exactly when $r=0$. If $r\ne0$, a constant diagonal gauge identifies $U_r$ with $U_1$. To check indecomposability directly, a horizontal endomorphism $T$ of $U_r$ satisfies
$$
 \theta T=r(NT-TN).
$$
The operator $T\mapsto r(NT-TN)$ on $\operatorname{Mat}_2(K)$ is nilpotent. Laurent coefficient comparison therefore makes $T$ constant. Such a constant endomorphism is of the form $b\I+cN$. The only idempotents of this form are $0$ and $\I$, proving indecomposability.
\end{proof}

The normal form realizes the standard description of extensions by de Rham cohomology \cite[\S4.4]{KedOverview}: its class is $r[dq/q]$. After adjoining $g^{1/2}$, a formal fundamental pair is
$$
 g^{-1/2},\qquad g^{-1/2}(H+r\log q).
$$
Here $\log q$ denotes an adjoined symbol with $\theta\log q=1$. It belongs to no Robba ring in $q$. In particular, the vanishing residue condition is needed for a Robba-ring developing map, but is unnecessary for defining $M(g)$.

\subsection{Frobenius and differential isomorphism classes}

The extension \eqref{eq:normal} carries the class $r[dq/q]$ in $H^1_{\mathrm{dR}}(\RK)$. Frobenius pullback multiplies $[dq/q]$ by $d$. The upper triangular matrix in the next theorem records precisely this scaling. The remaining obstruction is whether the rank-one twist is isomorphic to its Frobenius pullback.

\begin{theorem}[All Frobenius structures]\label[theorem]{thm:frobenius}
For every prime $p$ and every $g\in\RK^\times$, the module $M(g)$ admits a relative Frobenius structure if and only if there exist $\zeta\in\RK^\times$ and $c_0\in K^\times$ with
\begin{equation}\label{eq:square}
 \zeta^2=c_0\,\frac{\varphi(g)}{g}.
\end{equation}
Fix one such $\zeta$, and use $r,H,P_g$ from Theorem~\ref{thm:normal}. Every Frobenius structure, and only such a structure, has the form
\begin{equation}\label{eq:all-frob}
 \Phi=P_g\,\zeta T\,\varphi(P_g)^{-1},
\end{equation}
where
\begin{equation}\label{eq:constant-T}
 \begin{cases}
 T\in\GL_2(K),&r=0,\\[1mm]
 T=\begin{pmatrix}b&c\\0&db\end{pmatrix},\quad
 b\in K^\times,\ c\in K,&r\ne0.
 \end{cases}
\end{equation}
In these formulas $\det\Phi=c_0\det T$.
\end{theorem}

\begin{proof}
In the normal-form basis, a Frobenius matrix $C$ must satisfy
\begin{equation}\label{eq:C}
 \theta C=(a-d\varphi(a))C+rNC-drCN.
\end{equation}
Put $\delta=a-d\varphi(a)$. The lower-left entry satisfies $\theta C_{21}=\delta C_{21}$. If this entry is nonzero, set $\zeta=C_{21}$. Otherwise $C$ is upper triangular and invertibility gives $C_{11}\ne0$; its equation is $\theta C_{11}=\delta C_{11}$, and we set $\zeta=C_{11}$. In either case,
$$
 \theta\left(\zeta^2\frac{g}{\varphi(g)}\right)=0.
$$
The expression is a nonzero element of $K$, so $\zeta^2=c_0\varphi(g)/g$. Its square is a unit, hence $\zeta$ is a unit. This proves necessity without any assumption about zeros of individual entries.

Conversely, \eqref{eq:square} gives $\theta \zeta=\delta \zeta$. On writing $C=\zeta T$, equation~\eqref{eq:C} reduces to
\begin{equation}\label{eq:T}
 \theta T=rNT-drTN.
\end{equation}
The operator $\mathcal L(T)=rNT-drTN$ is nilpotent: its cube is zero because left and right multiplication by $N$ commute and have square zero. If $T=\sum T_nq^n$, then $(n\,\mathrm{id}-\mathcal L)T_n=0$. For $n\ne0$ the operator is invertible, so $T$ is constant. The remaining equation is $rNT=drTN$. For $r=0$ it imposes no condition. For $r\ne0$ it gives $T_{21}=0$ and $T_{22}=dT_{11}$, which is exactly \eqref{eq:constant-T}. Transforming back gives \eqref{eq:all-frob}. Finally,
$$
 \det\Phi=\frac{g\,\zeta^2}{\varphi(g)}\det T=c_0\det T.
$$
\end{proof}

\begin{corollary}[Odd-prime classification]\label[corollary]{cor:four}
Suppose $p\ne2$. Every $M(g)$ admits Frobenius. With $g=cq^m v^2$, the gauge
$$
 \widehat P_g=v^{-1}P_g
$$
transforms its horizontal system into
\begin{equation}\label{eq:constant-normal}
 \theta Z=A_{m,r}Z,\qquad A_{m,r}=-\frac m2\I+rN.
\end{equation}
In this basis all Frobenius matrices are
\begin{equation}\label{eq:constant-normal-frob}
 q^{(d-1)m/2}T,
\end{equation}
with $T$ as in \eqref{eq:constant-T}. Two modules $M(g_1)$ and $M(g_2)$ are isomorphic as differential modules exactly when
$$
 m(g_1)\equiv m(g_2)\pmod2,
 \qquad
 \bigl(\res(g_1\,dq/q)=0\bigr)\Longleftrightarrow
 \bigl(\res(g_2\,dq/q)=0\bigr).
$$
All four possibilities occur.
\end{corollary}

\begin{proof}
Since $a=-m/2-\theta v/v$, scalar multiplication of $P_g$ by $v^{-1}$ gives \eqref{eq:constant-normal}. Frobenius changes this constant matrix to $dA_{m,r}$. Coefficient comparison, exactly as in \eqref{eq:T}, yields \eqref{eq:constant-normal-frob}; its exponent is integral for odd $p$.

A gauge $C$ between two constant normal forms satisfies
$$
 \theta C=\frac{m_2-m_1}{2}C+r_1NC-r_2CN.
$$
The last two terms define a nilpotent operator. If $(m_2-m_1)/2\notin\mathbb Z$, every Laurent coefficient of $C$ vanishes. If this number is an integer $k$, every solution is $q^kC_0$ with $C_0$ constant and $r_1NC_0=r_2C_0N$. An invertible $C_0$ exists precisely when the two nilpotent matrices have the same rank. This proves the criterion. Representatives of the four classes are supplied by
$$
 g=q^2,\qquad g=q,\qquad g=1,\qquad g=q+p,
$$
respectively for even/split, odd/split, even/nonsplit, and odd/nonsplit.
\end{proof}

\begin{remark}[The nonsplit Frobenius ratio]\label[remark]{rem:ratio}
When $r\ne0$, the line $\ker N$ in the unipotent normal form is preserved by every differential automorphism and every Frobenius structure, by the matrix calculations above. After the quadratic coordinate extension $q=z^2$, with $\theta=\frac12z\,d/dz$, multiplication of the normal-form basis by $z^{-m}$ removes the scalar term. Frobenius becomes the constant matrix $T$ in \eqref{eq:constant-T}, with eigenvalues $b$ and $db$. Thus the ratio $d$ follows directly from the nonzero extension class. No slope-filtration hypothesis is needed for this conclusion.
\end{remark}

\begin{example}[The prime $2$]\label{ex:two}
For $p=2$, $d=2$, and $g=q$, one has $F_g=-1/2$ and $\varphi(g)/g=q$. Equation~\eqref{eq:square} has no solution: the winding number of a square is even. Hence $\theta^2y-y/4=0$ has no relative Frobenius structure over $\mathcal R_{\mathbb Q_2}$. The square-class criterion, rather than the automatic odd-prime conclusion, applies at $2$.
\end{example}

\begin{corollary}[Frobenius structures up to isomorphism]\label[corollary]{cor:frob-iso}
Fix an odd prime and one constant normal form $A_{m,r}$. Two Frobenius structures on this fixed differential module are called isomorphic if a differential automorphism intertwines their Frobenius maps. If $r=0$, their classes are the ordinary conjugacy classes of $T\in\GL_2(K)$. If $r\ne0$, every class has a unique representative
\[
 q^{(d-1)m/2}\diag(b,db),\qquad b\in K^\times.
\]
\end{corollary}

\begin{proof}
A differential automorphism of $A_{m,r}$ is constant, by the same Laurent coefficient calculation as above. For $r=0$ any constant invertible matrix is allowed. For $r\ne0$ it has the form $xI_2+yN$, $x\ne0$. Since $\varphi$ fixes constants, the action on $T$ is ordinary conjugation. Conjugating $\left(\begin{smallmatrix}b&c\\0&db\end{smallmatrix}\right)$ changes $c$ to $c+(d-1)by/x$, so it can be made zero, while $b$ remains unchanged.
\end{proof}

\subsection{Developing-map rigidity}

A Frobenius matrix is allowed to depend on $q$. A constant fractional linear functional equation for a developing map is a separate condition, governed by the following elementary obstruction.

\begin{proposition}[Developing-map rigidity]\label[proposition]{thm:rigidity}
For any prime $p$ and $d=p^s>1$, there are no $h\in\RK$ with $\theta h\in\RK^\times$ and $C\in\PGL_2(K)$ such that $\varphi_d(h)=C\circ h$. Furthermore, a constant invertible solution of \eqref{eq:normalized-gauge} forces $F=0$.
\end{proposition}
\begin{proof}
\cref{prop:formalism} and $\theta\varphi_d=d\varphi_d\theta$ give
\[
 F=\Sch{h}{\theta}=\Sch{\varphi_d(h)}{\theta}=d^2\varphi_d(F).
\]
For $F=\sum a_nq^n$, this identity gives $a_n=0$ when $d\nmid n$ and $a_{dn}=d^2a_n$. Every nonzero integer can be divided by $d$ only finitely many times, so all nonconstant coefficients vanish. Since $a_0=d^2a_0$, the constant coefficient vanishes as well.

To exclude $\Sch{h}{\theta}=0$, set $g=\theta h$. Its residue is zero, and \cref{thm:normal} identifies $M(g)$ with $L_g\oplus L_g$. Any nonzero horizontal solution $z$ of $L_g$ satisfies $z^2g\in K^\times$, hence is a unit. The ratio of two such solutions is constant, so $L_g\oplus L_g$ has either zero or two independent horizontal sections. By contrast, the solutions in $\RK$ of $\theta^2y=0$ are exactly the constants: first $\theta y$ is constant, and then its residue makes that constant zero. This gives a contradiction.

Finally, write a constant normalized gauge as $G=\left(\begin{smallmatrix}\alpha&\beta\\\gamma&\delta\end{smallmatrix}\right)$. The equation $B_FG=GB_{d^2\varphi_d(F)}$ gives $\delta=\alpha$ and
\[
 \alpha(F-d^2\varphi_d(F))=\beta(F-d^2\varphi_d(F))=0.
\]
Its first row is nonzero, so the same Laurent argument applies.
\end{proof}

For instance, $h=q$ has $\Sch{h}{\theta}=-1/2$ and produces the odd split class at odd primes. Its module admits Frobenius, although no constant fractional linear transformation sends $q$ to $q^d$. The proposition concerns an exact functional equation in $\RK$; it does not assert geometric rigidity of a connection in the sense of deformation theory.

\section{Kummer developing differentials}\label{sec:kummer}
The unit case has integral annular powers. Modular forms with a character naturally lead to $q^\nu f(q)$, where $\nu$ is rational and $f\in\RK^\times$. If $e\nu\in\mathbb Z$, the expression becomes an analytic function after the coordinate extension $q=z^e$; such an extension is called a \emph{Kummer extension}. The logarithmic derivation extends by $\theta(z)=z/e$. Even before this extension, the functions
\begin{equation}\label{eq:kummer-F}
 u_{\nu,f}=\nu+\frac{\theta f}{f},\qquad
 F_{\nu,f}=\theta u_{\nu,f}-\frac12u_{\nu,f}^2
\end{equation}
belong to $\RK$. We write $M_{\nu,f}$ for their Schwarzian differential module.

If $\nu\in\mathbb Z$, the expression $q^\nu f$ is already a unit of $\RK$, and Section~\ref{sec:annular} applies. The nonintegral case has no residue obstruction, because integration is replaced by a twisted equation with no zero denominator. The following lemma is the rational-exponent variant of annular integration \cite[Lemma~9.2.1]{Kedlaya}; the regular-singular background is \cite[\S7.3]{Kedlaya}.

\begin{lemma}[Twisted integration]\label[lemma]{lem:twisted}
For $\nu\in\mathbb Q\setminus\mathbb Z$, the operator $\theta+\nu$ is a bijection of $\RK$, with inverse
\[
 (\theta+\nu)^{-1}\left(\sum f_nq^n\right)
       =\sum\frac{f_n}{n+\nu}q^n.
\]
This inverse preserves convergence on every open annulus on which the input converges.
\end{lemma}

\begin{proof}
Write $\nu=a/b$ with integers $a,b$, $b\ne0$. No $bn+a$ is zero. Since
\[
 |(n+\nu)^{-1}|_p=|b|_p|bn+a|_p^{-1}
 \le |b|_p\bigl(|b|_\infty|n|_\infty+|a|_\infty\bigr),
\]
the denominators have at most polynomial growth. The two-radius argument of \cref{lem:integration} proves convergence. Coefficient comparison proves bijectivity.
\end{proof}

\begin{theorem}[Nonintegral Kummer classification]\label[theorem]{thm:kummer}
Suppose $p$ is odd, $\nu\in\mathbb Q\setminus\mathbb Z$, and $f=cq^m v^2\in\RK^\times$ as in \cref{lem:units}. Put
\[
 H=(\theta+\nu)^{-1}f,\qquad a=-\frac12u_{\nu,f},
 \qquad
 P_{\nu,f}=v^{-1}\begin{pmatrix}1&H\\a&f+aH\end{pmatrix}.
\]
Then $\det P_{\nu,f}=cq^m$, and this gauge transforms $M_{\nu,f}$ into
\begin{equation}\label{eq:kummer-normal}
 \theta Z=\Lambda Z,
 \qquad \Lambda=\diag(\lambda_1,\lambda_2),\qquad
 \lambda_1=-\frac{\nu+m}{2},\quad
 \lambda_2=\frac{\nu-m}{2},
\end{equation}
Equivalently,
\[
M_{\nu,f}\simeq L_{\lambda_1}\oplus L_{\lambda_2},
\]
where $L_\lambda$ denotes the rank-one differential module
$\theta z=\lambda z$. Moreover, for $d=p^s$, a Frobenius structure exists exactly when
\begin{equation}\label{eq:kummer-criterion}
 (d-1)\nu/2\in\mathbb Z\quad\text{or}\quad(d+1)\nu/2\in\mathbb Z.
\end{equation}
All such structures are
\begin{equation}\label{eq:kummer-frob}
 \Phi=P_{\nu,f}C\varphi_d(P_{\nu,f})^{-1},
 \qquad
 C_{ij}=\begin{cases}
 c_{ij}q^{\lambda_i-d\lambda_j},&\lambda_i-d\lambda_j\in\mathbb Z,\\
 0,&\lambda_i-d\lambda_j\notin\mathbb Z,
 \end{cases}
\end{equation}
where $c_{ij}\in K$ and $C$ is invertible. In the first alternative $C$ is diagonal; in the second it is antidiagonal. These alternatives cannot both hold for nonintegral $\nu$.
\end{theorem}

\begin{proof}
Let $P_0$ denote the matrix before multiplication by $v^{-1}$. The equations
\[
 \theta H=f-\nu H,\quad
 \theta f=-(2a+\nu)f,\quad
 \theta a+a^2=-F_{\nu,f}/2
\]
give
\[
 \theta P_0=B_{F_{\nu,f}}P_0-P_0\diag(a,a+\nu).
\]
Since $a=-(\nu+m)/2-\theta v/v$, the scalar factor $v^{-1}$ changes the last diagonal matrix to \eqref{eq:kummer-normal}. The determinant is $f/v^2=cq^m$, so this is an isomorphism over the original Robba ring.

In the diagonal basis, horizontality becomes
\[
 \theta C_{ij}=(\lambda_i-d\lambda_j)C_{ij}.
\]
\cref{lem:integration} gives \eqref{eq:kummer-frob}. An invertible two-by-two matrix requires either its diagonal pair or its antidiagonal pair to be nonzero. The diagonal exponents are $(d-1)(m+\nu)/2$ and $(d-1)(m-\nu)/2$; the antidiagonal exponents are $(d-1)m/2\mp(d+1)\nu/2$. Because $d$ is odd, their integrality is exactly \eqref{eq:kummer-criterion}. If both alternatives held, subtracting them would make $\nu$ integral. Conversely, in either alternative, choosing both available constants nonzero gives an invertible matrix.
\end{proof}

The normal form proves that the two rank-one constituents are distinct: their exponent difference is $\nu\notin\mathbb Z$. Here an exponent means the scalar $\lambda$ in $\theta z=\lambda z$, considered modulo integers because multiplication by $q^k$ shifts it by $k$; this is the shearing convention of \cite[Proposition~7.3.10]{Kedlaya}. Frobenius multiplies exponents by $d$. Criterion~\eqref{eq:kummer-criterion} says precisely that this multiplication preserves the unordered pair of exponent classes.

\begin{corollary}[Least Frobenius iterate]\label[corollary]{cor:least-power}
Under Theorem~\ref{thm:kummer}, let $b$ be the positive denominator of $\nu/2$ in lowest terms. A Frobenius structure for some $q\mapsto q^{p^s}$ exists if and only if $p\nmid b$. When it exists, the least exponent $s$ is
\[
 \min\{s\ge1:p^s\equiv1\text{ or }-1\pmod b\}.
\]
\end{corollary}

\begin{proof}
For a reduced fraction $\nu/2=a/b$, condition \eqref{eq:kummer-criterion} is $b\mid p^s-1$ or $b\mid p^s+1$. This is impossible if $p\mid b$. Otherwise a positive power of $p$ is $1$ in the finite group $(\mathbb Z/b\mathbb Z)^\times$.
\end{proof}

\begin{corollary}[Square units at every prime]\label[corollary]{cor:kummer-square}
Let $p$ be any prime, $\nu\in\mathbb Q\setminus\mathbb Z$, and
$w\in\RK^\times$. Put $f=w^2$, $H=(\theta+\nu)^{-1}f$, and
$a=-\nu/2-\theta w/w$. The gauge
\[
 P_{\nu,w^2}=w^{-1}
 \begin{pmatrix}1&H\\a&w^2+aH\end{pmatrix}
\]
has determinant one and transforms $M_{\nu,w^2}$ into
\[
 \theta Z=\diag(-\nu/2,\nu/2)Z.
\]
For every $d=p^s$, this module admits Frobenius exactly when
\eqref{eq:kummer-criterion} holds. All Frobenius structures are
given by \eqref{eq:kummer-frob}, with
$\lambda_1=-\nu/2$ and $\lambda_2=\nu/2$.
\end{corollary}

\begin{proof}
The gauge calculation in \cref{thm:kummer} applies with $v=w$,
$m=0$, and $c=1$. It uses no binomial square root, so it is valid
also at $p=2$. Its determinant is $f/w^2=1$. Laurent coefficient
comparison gives the Frobenius entries and criterion exactly as
in that theorem, now without a winding-number term.
\end{proof}

\subsection{Constant coefficients and the cusp test}
The Euler equations are the basic constant regular-singular systems \cite[\S7.3]{Kedlaya}. The following direct calculation records their Frobenius criterion in our conventions.
\begin{proposition}[Euler equations]\label[proposition]{prop:euler}
Let $\alpha\in K$ and $E_\alpha$ be $\theta^2y-\alpha^2y=0$. If $\alpha\ne0$, it admits $\varphi_d$-Frobenius exactly when
\[
 (d-1)\alpha\in\mathbb Z\quad\text{or}\quad(d+1)\alpha\in\mathbb Z.
\]
For $\alpha=0$, it admits Frobenius and is the nonsplit unipotent module $U_1$.
\end{proposition}
\begin{proof}
For $\alpha\ne0$, the constant matrix
\[
 S_\alpha=\begin{pmatrix}1&1\\\alpha&-\alpha\end{pmatrix}
\]
diagonalizes the companion matrix to $\diag(\alpha,-\alpha)$. A Frobenius matrix in this basis has entries proportional to monomials with exponents $(1-d)\alpha,(1+d)\alpha,-(1+d)\alpha,(d-1)\alpha$, whenever these exponents are integers. An invertible such matrix exists precisely when both diagonal
exponents or both antidiagonal exponents are integers. For $\alpha=0$, the companion matrix is $N$, and $\diag(1,d)$ gives Frobenius.
\end{proof}

For a coefficient $F\in K[[q]]$ with $F(0)=-2\alpha^2$, the same congruence is a necessary condition if a Frobenius gauge belongs to $\GL_2(K((q)))$. Indeed, regular-singular exponents over $K((q))$ are invariant modulo integral shifts, and Frobenius multiplies the pair $\alpha,-\alpha$ by $d$; see \cite[Propositions~7.3.10 and 7.3.12]{Kedlaya}. This meromorphic cusp test does not by itself apply to arbitrary Robba-ring gauges, which may have infinitely many negative powers. Theorems~\ref{thm:kummer} and \ref{thm:modular-family} establish a gauge on the full annular germ and hence give both necessity and sufficiency there.

\section{Stieltjes differentials and projective connections}\label{sec:stieltjes}
\subsection{The classical construction}
Classical Heine--Stieltjes theory starts with polynomials $R_0,R_1$ and seeks a polynomial $R_2$ for which $R_0Q''+R_1Q'+R_2Q=0$ has a polynomial solution $Q$ of prescribed degree. In the usual Fuchsian setting, $R_0$ has simple roots, $\deg R_1\le\deg R_0-1$, and $\deg R_2\le\deg R_0-2$. The unknown $R_2$ is the Van Vleck polynomial. At a simple root of $Q$, division by $R_0Q'$ gives an equation involving the other roots of $Q$. These are the Stieltjes equations; see \cite{Stieltjes} and \cite[\S1.2, Theorem~A]{Scherbak}.

Writing $a=-R_1/R_0$ and formally choosing $A$ with $A'/A=a$ recasts the equation as $Q''-aQ'+VQ=0$, where $V=R_2/R_0$. Reduction of order associates to it the differential $AQ^{-2}dx$. The absence of logarithms at the roots of $Q$ is exactly residue cancellation for this differential. We first treat rational $A$, for which this construction is global on $\mathbb P^1$. The rational logarithmic-derivative version needed for modular characters is stated immediately afterwards.

In this section $K$ is any characteristic-zero field. Let $A\in K(x)^\times$, and let $Q=\prod_{i=1}^n(x-x_i)$ be monic and squarefree, with every root avoiding the divisor of $A$. Roots may be taken in an algebraic closure. Write
\begin{equation}\label{eq:Stieltjes}
 a_A=\frac{A'}A,\qquad
 S_i=a_A(x_i)-2\sum_{j\ne i}\frac1{x_i-x_j},\qquad
 V=-\frac{Q''-a_AQ'}Q.
\end{equation}
The following elementary formulation of the classical correspondence is compatible with the projective-structure account in \cite[\S2.1]{Korotkin}. We call $Q$ a \emph{Stieltjes polynomial for $A$} when the equivalent conditions below hold. The function $V$ is a rational Van Vleck term; multiplication by $R_0$ gives the classical polynomial \cite[\S1.2, Theorem~A]{Scherbak}.

\begin{proposition}[Residues and projective cancellation]\label[proposition]{prop:stieltjes}
The following conditions are equivalent:
\begin{enumerate}[label=\textup{(\roman*)}]
\item $\omega_Q=AQ^{-2}dx$ has zero residue at every $x_i$;
\item $S_i=0$ for every $i$;
\item the rational function $V$ in \eqref{eq:Stieltjes} is regular at every $x_i$;
\item the coefficient
$$
 \mathcal F_{A,Q}=
 \left(a_A-2\frac{Q'}Q\right)'
 -\frac12\left(a_A-2\frac{Q'}Q\right)^2
$$
is regular at every $x_i$.
\end{enumerate}
In this case
\begin{equation}\label{eq:vv}
 Q''-a_AQ'+VQ=0,\qquad
 \mathcal F_{A,Q}=2V+a_A'-\frac12a_A^2.
\end{equation}
The finite poles of both $V$ and $\mathcal F_{A,Q}$ lie among the zeros and poles of $A$.
\end{proposition}

\begin{proof}
We use the local residue calculation underlying the Stieltjes--Bethe equations \cite[\S2.1]{Korotkin}. Put $Q=(x-x_i)Q_i$. Expansion at $x_i$ gives
$$
 \res_{x_i}\omega_Q=
 \left(\frac{A}{Q_i^2}\right)'(x_i)
 =\frac{A(x_i)}{Q'(x_i)^2}S_i.
$$
The prefactor is nonzero. Since $Q''(x_i)/Q'(x_i)=2\sum_{j\ne i}(x_i-x_j)^{-1}$, residue vanishing is equivalent to cancellation of the numerator of $V$ at each simple root of $Q$. This proves the first three equivalences.

Set $w=Q'/Q$. Then $w'=a_Aw-V-w^2$, and expansion of $(a_A-2w)'-(a_A-2w)^2/2$ gives the second identity in \eqref{eq:vv}. Because $a_A$ is regular at $x_i$, this identity proves the fourth equivalence and the pole assertion.
\end{proof}

The regularity condition on $V$ is essential: existence of some rational $V$ satisfying the displayed differential equation is automatic for every $Q$.

\begin{remark}[Rational logarithmic derivatives]\label[remark]{rem:weighted}
The root equations and the identity in \eqref{eq:vv} make sense for an arbitrary rational function $a(x)$, with $a_A$ replaced by $a$. Around a root $x_i$ outside its poles, choose a formal unit $A_i$ with $A_i'/A_i=a$ and $A_i(x_i)=1$. The same residue calculation applies to $A_iQ^{-2}dx$. Thus residue cancellation is a local condition and does not require a global rational $A$. When $a=\sum_j\gamma_j/(x-z_j)$ with rational $\gamma_j$, one can take $A=\prod_j(x-z_j)^{\gamma_j}$ on a finite algebraic cover. We call these \emph{weighted Stieltjes data}; the rational logarithmic-derivative formulation is classical \cite[\S1.2]{Scherbak}. The later global statements about rational differentials retain their rationality hypothesis. For Stieltjes equations with an arbitrary
rational external-field derivative, see \cite{BCG}.
\end{remark}

\subsection{Local exponents and global restrictions}

For $y''+\mathcal F(s)y/2=0$ with $\mathcal F(s)=c_{-2}s^{-2}+O(s^{-1})$, the local exponents are the roots of $z(z-1)+c_{-2}/2=0$. Their difference is defined up to sign. This is the regular-singular convention of \cite[\S7.2]{Kedlaya}. We now compute it directly from a developing differential.

\begin{remark}[Ordinary points at Stieltjes roots]
At a Stieltjes root $x_i$, the developing differential has a
double pole and zero residue. Its formal primitive therefore
has a simple pole, whose reciprocal is a formal local
coordinate. Thus $x_i$ is an ordinary point of the associated
projective connection, as expressed by
Proposition~\ref{prop:stieltjes}.
\end{remark}
More generally, if a rational developing differential has local
order $j$, its projective coefficient has leading term
$-j(j+2)/(2s^2)$, so the associated equation is regular singular
with exponent difference $j+1$, up to sign
\cite[\S7.2]{Kedlaya}. Over $\mathbb C$, primitives of rational
differentials have translation monodromy, which is trivial
precisely when all residues vanish; compare
\cite[\S2.1.3]{Korotkin}. These properties restrict which
projective connections arise from rational differentials.
For example, the Legendre hypergeometric equation with parameters
$(1/2,1/2;1)$ has zero exponent differences at $0,1,\infty$
\cite[\S15.10(i)]{DLMF}. A rational developing differential
would therefore have order $-1$ at these three points and
order $0$ or $-2$ elsewhere, contradicting the divisor degree
$-2$ on $\mathbb P^1$. This restriction concerns global rational
developing differentials and does not preclude an annular
analytic presentation.
\subsection{Wronskian pencils and nondegeneracy}

The rank-two reproduction construction
\cite[\S3, Lemmas~3.2--3.5]{MV} explains how moving Stieltjes
roots can represent a fixed projective connection. It gives
a useful restriction on the polynomial configurations to which
the nondegenerate deformation argument applies.

Let $A\in K[x]\setminus\{0\}$ have degree $D$, and let
$Q\in K[x]$ be a Stieltjes polynomial of degree $n\ge1$.
Since $AQ^{-2}dx$ has zero residues at all its finite poles,
partial fractions give a rational primitive $P/Q$, with
$P\in K[x]$. Thus
\begin{equation}\label{eq:wronski}
P'Q-PQ'=A.
\end{equation}
If $D\le2n-2$, the primitive can be chosen to vanish at infinity,
so that $\deg P<n$. The polynomials
$Q_t=Q+tP$
are then monic of degree $n$ and satisfy
\[
\left(\frac{P}{Q_t}\right)'=\frac{A}{Q_t^2}.
\]
Consequently, $Q_t$ remains Stieltjes whenever its roots are
simple and avoid the zeros of $A$, conditions holding on a
Zariski-open neighborhood of $t=0$.

Writing $x_i$ for the roots of $Q$, differentiation at $t=0$
gives the root tangent
\[
v_i=-\frac{P(x_i)}{Q'(x_i)}.
\]
This vector is nonzero because $P\ne0$ and $\deg P<n$.
Differentiating the Stieltjes equations along the pencil gives
$Jv=0$, where $J=(\partial S_i/\partial x_j)$.
Hence the configuration is degenerate.
If $D=2n-1$, the coefficient of $x^{-1}$ in $A/Q^2$ is the
nonzero leading coefficient of $A$, contradicting the
vanishing residue at infinity. Thus a nondegenerate
Stieltjes configuration for polynomial $A$ requires $D\ge2n$.

For the pencil above, the developing maps satisfy
\[
h_t=\frac{P}{Q+tP}=\frac{h_0}{1+th_0}.
\]
Their Schwarzian coefficient is therefore independent of $t$
by projective invariance. This distinguishes motion of the
Stieltjes roots from variation of the associated projective
connection.

\section{Deformations and explicit families}\label{sec:families}
\subsection{A uniform annular formula}
Let $K/\mathbb Q_p$ be finite and let $U$ be a reduced affinoid space with algebra $B=\mathcal O(U)$. Concretely, $B$ is a quotient of a Tate algebra, equipped with a complete submultiplicative norm; see \cite[\S5.1]{BGR}. A relative Robba series is a Laurent series with coefficients in $B$ converging in a fixed Banach norm at every radius $\rho<t<1$, for one common $\rho$. We write $\R_B$ for the union of these rings over $\rho$; for the general theory of annular connections with parameters, see \cite{KedFamilies}. The derivation $\theta$ and $\varphi_d$ act on $q$ and fix $B$. All connections in this section are relative to $B$: no connection in the parameter directions is specified. This definition requires convergence up to the outer boundary in the Robba sense; an analytic function on a single closed annulus with outer radius less than $1$ does not supply such a germ.

The next theorem assumes a single dominant monomial uniformly in the parameters. This is an explicit condition on the Banach-valued series. Pointwise unit conditions alone are not substituted for it. Its role is to construct a square root and a gauge over the whole parameter algebra, including points where the extension splits.

\begin{theorem}[Uniform Frobenius in a family]\label[theorem]{thm:family}
Suppose $p\ne2$ and
\begin{equation}\label{eq:family-dom}
 g=cq^m(1+e)\in\R_B^\times,\qquad c\in B^\times,
 \qquad \|e\|_t<1\quad(\rho<t<1).
\end{equation}
Let $r\in B$ be its constant Laurent coefficient, let $H\in\R_B$ satisfy $\theta H=g-r$, and put
$$
 v=(1+e)^{1/2},\qquad \widehat P_g=v^{-1}P_g.
$$
Then $M(g)$ is analytically gauge equivalent to $\theta Z=A_{m,r}Z$, and
\begin{equation}\label{eq:family-frob}
 \Phi_g=\widehat P_g\,q^{(d-1)m/2}
       \diag(1,d)\,\varphi(\widehat P_g)^{-1}
\end{equation}
is a Frobenius structure over $\R_B$. The matrix is defined on the entire parameter space, including the locus $r=0$, and $\det\Phi_g=d$.
\end{theorem}

\begin{proof}
The estimates in Lemma~\ref{lem:integration}, with the Banach norm of $B$, give $H$ on the same open annulus. The binomial series converges uniformly on each closed subannulus, giving $v$ and $v^{-1}$. The matrix calculation in Theorem~\ref{thm:normal} takes place over $\R_B$ unchanged. Finally,
$$
 N\diag(1,d)=d\diag(1,d)N
$$
is valid without dividing by $r$, which proves horizontality of \eqref{eq:family-frob} even where $r$ vanishes. Since $\det\widehat P_g=cq^m$ and $\varphi(c)=c$, its determinant is $d$.
\end{proof}

The formula separates analytic Frobenius existence from the change in differential-module type at $r=0$. For example, for $|t|\le|p|$ the family $g_t=q+t$ has common winding number $1$, residue $t$, and a Frobenius matrix given by \eqref{eq:family-frob}. Its differential module splits precisely at $t=0$.

\subsection{Analytic deformation of the Stieltjes equations}

For the Stieltjes deformation problem, write $A_u(x)=P(u,x)/R(u,x)$ with $P,R\in B[x]$, and restrict to the configuration space where all denominators and numerators are invertible. The logarithmic differential of
$$
 \mathcal M(X)=\frac{\prod_iA_u(x_i)}{\prod_{i<j}(x_i-x_j)^2}
$$
has components $S_i(u,X)$ from \eqref{eq:Stieltjes}. This is the reciprocal convention for the usual rank-two master function, which has the same critical set; see \cite[\S1.2]{Scherbak} and \cite[\S2]{MV}. Here ``logarithmic differential'' means $d\mathcal M/\mathcal M$, so no analytic branch of a logarithm is chosen. The Jacobian $J=(\partial S_i/\partial x_j)$ has entries
\begin{equation}\label{eq:hessian}
 J_{ij}=\begin{cases}
 a_{A_u}'(x_i)+2\displaystyle\sum_{\ell\ne i}(x_i-x_\ell)^{-2},&i=j,\\[1mm]
 -2(x_i-x_j)^{-2},&i\ne j.
 \end{cases}
\end{equation}
A configuration is nondegenerate when $\det J\ne0$. For $X=(x_1,\ldots,x_n)$, define
$S(u,X)=\bigl(S_1(u,X),\ldots,S_n(u,X)\bigr)$. Thus, $X$ is a Stieltjes configuration for $A_u$ precisely when
$S(u,X)=0$.

The following proposition combines the local analytic deformation of
nondegenerate Stieltjes configurations with the annular Frobenius
criterion developed above.
\begin{proposition}[Stieltjes deformation with annular Frobenius]\label[proposition]{thm:stieltjes-family}
Let $u_0\in U(K)$ and let $X_0\in K^n$ be a nondegenerate Stieltjes configuration for $A_{u_0}$. After restricting to a neighborhood of $(u_0,X_0)$, there is a unique analytic branch $X(u)$ in a sufficiently small polydisc around $X_0$ satisfying $X(u_0)=X_0$ and $S(u,X(u))=0$. The terms $V_u$ and $\mathcal F_{A_u,Q_u}$ in \eqref{eq:vv} depend analytically on $u$ and are regular at every moving root.

Let $t\in\mathcal R_K$ be fixed, with $\theta t\in\mathcal R_K^\times$, and assume that \[
P(u,t),\qquad R(u,t),\qquad Q_u(t)
\]
are units in $\mathcal R_B$. Then $g_u=\frac{A_u(t)}{Q_u(t)^2}\theta t\in \mathcal R_B^\times$ and its Schwarzian coefficient is
\begin{equation}\label{eq:pullback-F}
 F_{g_u}=(\theta t)^2\mathcal F_{A_u,Q_u}(t)+\Sch{t}{\theta}.
\end{equation}
If $p\ne2$, each fiber admits Frobenius. If, in addition, the family satisfies \eqref{eq:family-dom}, formula~\eqref{eq:family-frob}
gives an analytic Frobenius structure over $\mathcal R_B$. A primitive $h_u\in\R_B$ with $\theta h_u=g_u$ exists exactly when the constant Laurent coefficient of $g_u$ is zero in $B$.
\end{proposition}
\begin{proof}
For the analytic implicit-function argument, let $J_0=\partial_XS(u_0,X_0)$, which is invertible by
nondegeneracy, and set
\[
T_u(Z)=Z-J_0^{-1}S(u,X_0+Z).
\]
Choose an affinoid neighborhood of $u_0$ and a closed polydisc
around $X_0$ on which the denominators in $S$ are invertible.
Then $T$ is analytic,
$T_{u_0}(0)=0,
\text{ and }
\partial_ZT_{u_0}(0)=0$.
We may therefore choose a sufficiently small radius $r>0$ and
shrink the parameter neighborhood so that
\[
\|T(0)\|\le r,
\qquad
\|T(Z)-T(W)\|\le c\|Z-W\|,
\qquad 0<c<1,
\]
whenever $\|Z\|,\|W\|\le r$. Here we continue to denote the
restricted affinoid neighborhood by $U$, set $B=\mathcal O(U)$,
and equip $B^n$ with the maximum of a complete submultiplicative
affinoid norm on its coordinates. To obtain these bounds, expand
$T$ in the root variables. Shrinking $r$ controls the
higher-degree terms, while shrinking $U$ makes the linear
coefficients sufficiently small and the constant term bounded
by $r$. These coefficient bounds hold uniformly over $B$. Since $B$ is an affinoid $K$-algebra, it is complete for an
affinoid Banach-algebra norm; see
\cite[\S\S5.1.1 and 6.1.1]{BGR}.

The ultrametric inequality shows that $T$ maps the closed ball
of radius $r$ in $B^n$ into itself. Starting with $Z_0=0$, define
$Z_{k+1}=T(Z_k)$.
The contraction estimate gives
\[
\|Z_{k+1}-Z_k\|
\le c^k\|Z_1-Z_0\|,
\]
so completeness yields a limit $Z\in B^n$. This limit is the
unique fixed point in the ball. Consequently,
\[
X(u)=X_0+Z(u)
\]
is analytic and satisfies $S(u,X(u))=0$. Evaluating the iteration
at $u_0$ gives $Z_k(u_0)=0$ for every $k$, hence
$X(u_0)=X_0$. The same uniform contraction bounds give uniqueness
among analytic branches taking values in the chosen polydisc.

Write $X(u)=(x_1(u),\ldots,x_n(u))$ and
\[
Q_u(x)=\prod_{i=1}^n(x-x_i(u))\in B[x].
\]
After shrinking $U$ if necessary, the elements
\[
P(u,x_i(u)),\qquad R(u,x_i(u)),\qquad
x_i(u)-x_j(u)\quad(i\ne j)
\]
are units in $B$. In particular, the roots remain distinct and
avoid the zeros and poles of $A_u$.

We next verify cancellation at the moving roots over $B$.
With primes denoting differentiation in $x$, put
\[
D=PR,\qquad L=P'R-PR',
\qquad
a_{A_u}=\frac{A_u'}{A_u}=\frac{L}{D}.
\]
The Stieltjes equations and the identity
\[
\frac{Q_u''(x_i(u))}{Q_u'(x_i(u))}
=2\sum_{j\ne i}\frac{1}{x_i(u)-x_j(u)}
\]
imply that the polynomial
$W=DQ_u''-LQ_u'$
vanishes at every $x_i(u)$. Division by the monic polynomial
$Q_u$ gives
\[
W=Q_u\widetilde W+H,
\qquad
\deg H<n,
\]
with $\widetilde W,H\in B[x]$. Since $H(x_i(u))=0$ for every
$i$, and the Vandermonde determinant
\[
\prod_{i<j}(x_j(u)-x_i(u))
\]
is a unit in $B$, all coefficients of $H$ vanish. Thus
\[
V_u=\frac{a_{A_u}Q_u'-Q_u''}{Q_u}
=-\frac{\widetilde W}{D}.
\]
This expression depends analytically on $u$ and is regular at
each moving root. Formula~\eqref{eq:vv}, equivalently
\[
\mathcal F_{A_u,Q_u}
=a_{A_u}'-\frac12a_{A_u}^2+2V_u,
\]
gives the same conclusions for $\mathcal F_{A_u,Q_u}$.

For the annular pullback, the unit assumptions give
$g_u=P(u,t)\theta t/[R(u,t)Q_u(t)^2]\in\R_B^\times$.
To verify the chain rule without choosing a primitive, set
$b_u=a_{A_u}-2Q_u'/Q_u$, $v=\theta t$, and $c=\theta v/v$.
Then $\theta g_u/g_u=v\,b_u(t)+c$, so the mixed terms cancel:
\[
 F_{g_u}=v^2\left(b_u'(t)-\frac12b_u(t)^2\right)
             +\theta c-\frac12c^2.
\]
The two terms are $(\theta t)^2\mathcal F_{A_u,Q_u}(t)$ and
$\Sch{t}{\theta}$, respectively, proving \eqref{eq:pullback-F}.

For $p\ne2$, specialization preserves the unit $g_u$, and
\cref{cor:four} supplies Frobenius on each fiber. Under the
additional uniform hypothesis \eqref{eq:family-dom},
\cref{thm:family} supplies \eqref{eq:family-frob} over $\R_B$.

Finally, if $g_u=\sum g_mq^m$, a primitive requires $g_0=0$.
When this holds, the series
\[
 h_u=\sum_{m\ne0}\frac{g_m}{m}q^m
\]
converges in $\R_B$ by the Banach-valued version of
\cref{lem:integration} used in \cref{thm:family}, and
$\theta h_u=g_u$.
\end{proof}

\begin{remark}[What may move]
If the complete divisor of the rational function $A_u$ is fixed on $\mathbb P^1$, then $A_u$ differs from a fixed rational function by a scalar. Its logarithmic derivative is constant in $u$, so the unique nondegenerate Stieltjes branch and its projective coefficient are constant. Nontrivial deformations of a nondegenerate branch therefore require variation of the divisor of $A_u$. Throughout \cref{thm:stieltjes-family}, the cancellation statement concerns the moving roots of $Q_u$; the divisor of $A_u$ may move.
\end{remark}

\begin{example}[A nonconstant family with a moving Stieltjes root]\label{ex:moving}
Let
$$
 A_{a,b}(x)=(x-a)^2+b,\qquad Q_a(x)=x-a,\qquad b\ne0.
$$
The root $x=a$ is Stieltjes and its Jacobian is $2/b$. Moreover,
$$
 \frac{A_{a,b}}{Q_a^2}dx=d\left(x-a-\frac{b}{x-a}\right),
 \qquad
 \mathcal F_{A_{a,b},Q_a}(x)=\frac{6b}{((x-a)^2+b)^2}.
$$
With $t(q)=q$, one obtains
\begin{equation}\label{eq:moving-explicit}
 \begin{split}
 H_{a,b}(q)&=q-\frac{b}{q-a},\\
 g_{a,b}(q)&=q\left(1+\frac{b}{(q-a)^2}\right),\\
 F_{a,b}(q)&=-\frac12+\frac{6bq^2}{((q-a)^2+b)^2}.
 \end{split}
\end{equation}
Fix a parameter affinoid with $|a|\le a_*<\rho<1$ and $0<b_-\le|b|\le b_*<\rho^2$. Then $g_{a,b}$ has winding number $1$ and zero residue on $\rho<|q|<1$, uniformly in the parameters. For every odd $p$, set
$$
 v_{a,b}=\left(1+\frac{b}{(q-a)^2}\right)^{1/2}.
$$
Substitution of $H_{a,b},g_{a,b},v_{a,b}$ into \eqref{eq:P} and \eqref{eq:family-frob} gives an explicit analytic Frobenius matrix. Thus this family has moving roots, nonconstant projective coefficients, and Frobenius throughout the parameter domain. \cref{thm:rigidity} excludes constant projective Frobenius equivariance of every $H_{a,b}$.
\end{example}

\begin{example}[A nondegenerate family across zero residue]\label{ex:residue-crossing}
Add a parameter $t$ to the preceding example by setting
\[
 A_{a,b,t}(x)=(x-a)^2\left(1+\frac{t}{x}\right)+b,
 \qquad Q_a(x)=x-a.
\]
Assume $a,b\ne0$ and $1+t/a\ne0$. At $x=a$, one has $A=b$, $A'=0$, and $A''=2(1+t/a)$, so this is a nondegenerate Stieltjes configuration with Jacobian $2(1+t/a)/b$. Its annular developing differential is
\[
 g_{a,b,t}\frac{dq}{q}
 =\left(1+\frac{b}{(q-a)^2}+\frac{t}{q}\right)dq,
 \qquad
 g_{a,b,t}=\theta\left(q-\frac{b}{q-a}\right)+t.
\]
Choose a parameter affinoid satisfying
\[
 0<a_-\le |a|\le a_*<\rho<1,\quad
 0<b_-\le|b|\le b_*<\rho^2,\quad
 |t|\le t_*<a_-.
\]
Then $g_{a,b,t}/q$ has norm distance less than one from $1$ on the common annulus. Its winding number is $1$, its residue is $t$, and the Stieltjes Jacobian never vanishes. For odd $p$, the Frobenius matrix of \cref{thm:family} is therefore analytic throughout this family, while the differential module is split precisely on $t=0$. At that locus the projective coefficient reduces to the nonconstant explicit expression \eqref{eq:moving-explicit}. This exhibits a change of annular extension class without degeneracy of the Stieltjes critical point.
\end{example}

\subsection{Jacobi configurations and a quantitative nondegeneracy test}

For the spectral calculation, allow rational parameters $\alpha,\beta>-1$ under their real embedding, and let $n\ge1$. Let $Q_n$ be the monic normalization of the Jacobi polynomial $P_n^{(\alpha,\beta)}$. Set
$$
 A(x)=(1-x)^{-\alpha-1}(1+x)^{-\beta-1},\qquad
 \lambda_k=k(k+\alpha+\beta+1).
$$
When the parameters are not integers, $A$ is interpreted through its rational logarithmic derivative as in Remark~\ref{rem:weighted}. The Jacobi equation \cite[\S18.8]{DLMF} is
$$
 (1-x^2)Q_n''+
 [\beta-\alpha-(\alpha+\beta+2)x]Q_n'+\lambda_nQ_n=0.
$$
Its roots are simple and lie in $(-1,1)$ under the real embedding \cite[\S18.16(ii)]{DLMF}. It is the Stieltjes equation for this $A$, with $V_n=\lambda_n/(1-x^2)$. Formula~\eqref{eq:vv} gives
\begin{equation}\label{eq:jacobi-F}
 \begin{split}
 \mathcal F_n(x)=\frac{1-\alpha^2}{2(1-x)^2}
 +\frac{1-\beta^2}{2(1+x)^2}
 +\frac{(\alpha+1)(\beta+1)+2\lambda_n}{1-x^2}.
 \end{split}
\end{equation}
The exponent differences at $1,-1,\infty$ are $\alpha,\beta,2n+\alpha+\beta+1$, and every root of $Q_n$ is an ordinary point of this projective coefficient.

The following Hessian identity is the Stieltjes-coordinate form of the spectral-gap description for perturbations of classical polynomial zeros; compare \cite[Theorem~3.1 and \S4.3]{Sasaki}. We include a proof to make the nondegeneracy test explicit.

\begin{proposition}[Jacobi Hessian spectrum]\label[proposition]{prop:jacobi-hessian}
Let $x_1,\ldots,x_n$ be the roots of $Q_n$, let $J$ be \eqref{eq:hessian} for the rational logarithmic derivative of $A$, and put $D=\diag(1-x_i^2)$. Then $DJ$ is similar to the operator $\mathcal L+\lambda_n$ on polynomials of degree less than $n$, where
$$
 \mathcal L=(1-x^2)\frac{d^2}{dx^2}
      +[\beta-\alpha-(\alpha+\beta+2)x]\frac{d}{dx}.
$$
Its eigenvalues are $\lambda_n-\lambda_k$, $0\le k<n$. Consequently,
\begin{equation}\label{eq:jacobi-det}
 \det J=
 \frac{\displaystyle n!\prod_{k=0}^{n-1}(n+k+\alpha+\beta+1)}
      {\displaystyle\prod_{i=1}^n(1-x_i^2)}\ne0.
\end{equation}
The same nonvanishing holds after every characteristic-zero $p$-adic embedding of a splitting field of $Q_n$.
\end{proposition}

\begin{proof}
We express the polynomial-zero perturbation argument of \cite[Theorem~3.1]{Sasaki} in root coordinates. A root tangent $v=(v_i)$ determines the polynomial $R$ of degree less than $n$ given by
$
 R(x_i)=-Q_n'(x_i)v_i.
$
This is an isomorphism by interpolation. For an arbitrary monic polynomial $Q$ of degree $n$, evaluate $E(Q)=(\mathcal L+\lambda_n)Q$ at its roots:
$$
 E(Q)(x_i)=-(1-x_i^2)Q'(x_i)S_i.
$$
At $Q_n$ the polynomial $E(Q_n)$ is identically zero. Differentiating this identity along the root tangent therefore gives
$$
 (\mathcal L+\lambda_n)R(x_i)
 =-(1-x_i^2)Q_n'(x_i)(Jv)_i.
$$
The interpolation isomorphism conjugates $DJ$ to $\mathcal L+\lambda_n$. On the monomial basis, $\mathcal L$ is triangular with diagonal $-\lambda_k$, $0\le k<n$. Hence the determinant of $DJ$ is
$$
 \prod_{k=0}^{n-1}(\lambda_n-\lambda_k)
 =n!\prod_{k=0}^{n-1}(n+k+\alpha+\beta+1),
$$
which proves \eqref{eq:jacobi-det}. All factors are nonzero algebraic numbers, so their images under field embeddings remain nonzero.
\end{proof}

\begin{corollary}[Annular type of the Jacobi family]\label[corollary]{cor:jacobi-type}
Suppose now that $\alpha,\beta$ are nonnegative integers, and fix $n$ and an odd prime. After a finite scalar extension containing the roots, choose $\varepsilon\in K^\times$ sufficiently small and pull back by $x=q/\varepsilon$. The resulting $g$ is a Robba-ring unit with residue zero and
$$
 m(g)=-(2n+\alpha+\beta+1).
$$
Its differential-module isomorphism class depends only on $\alpha+\beta+1$ modulo $2$. In particular, increasing $n$ does not change that annular class.
\end{corollary}

\begin{proof}
The Laurent expansion at infinity has the form
$$
 \frac{A(x)}{Q_n(x)^2}x
 =c\,x^{-(2n+\alpha+\beta+1)}(1+O(x^{-1})),\qquad c\ne0.
$$
Choose $|\varepsilon|$ small enough that every finite zero and pole, after scaling, lies strictly inside the inner boundary. Factoring the finitely many linear factors makes the term in parentheses a unit of norm distance less than $1$ from $1$ on the annulus. The expansion has only strictly negative powers of $q$, so its residue is zero. Corollary~\ref{cor:four} gives the result. For a fixed finite collection of degrees the scale can be chosen in common.
\end{proof}

\section{Modular pullbacks and least Frobenius iterates}\label{sec:modular}

\subsection{Genus-zero descent and characters}

We use the modular-form conventions of \cite[\S1.1]{Zagier}. Let $X_\Gamma$ be a genus-zero modular curve. A \emph{Hauptmodul} is a coordinate $t$ identifying its meromorphic function field with $L(t)$, after choosing a field of definition $L$. If $f_0$ is a meromorphic form of weight two with trivial character, defined over $L$, then $f_0/(\theta t)$ is invariant. Hence
\begin{equation}\label{eq:modular-A}
 f_0=A(t)\theta t,\qquad A\in L(x),
 \qquad
 \frac{f_0}{Q(t)^2}\frac{dq}{q}
 =t^*\left(\frac{A(x)}{Q(x)^2}\,dx\right).
\end{equation}
Here the cusp has width one, so $q=e^{2\pi i\tau}$ and $\theta=(2\pi i)^{-1}d/d\tau$. Other widths are handled by constant rescaling of the derivation. The modular terminology and genus-zero coordinate viewpoint follow \cite[\S\S1.1,5.4]{Zagier}.

\begin{corollary}[Trivial-character pullbacks]\label[corollary]{cor:modular}
Suppose the data in \eqref{eq:modular-A} are defined over a number field, $Q$ is Stieltjes for $A$, and a $p$-adic embedding gives
\[
 g=\frac{f_0}{Q(t)^2}\in\RK^\times.
\]
For odd $p$ the corresponding module admits $\varphi_d$-Frobenius for every $d=p^s$. A developing map $h\in\RK$ with $\theta h=g$ exists exactly when $\res(g\,dq/q)=0$. Every such map satisfies the obstruction in \cref{thm:rigidity}. Moreover, for odd $p$, analytic families of these developing differentials satisfying \eqref{eq:family-dom} carry the Frobenius matrix
\eqref{eq:family-frob}.
\end{corollary}
\begin{proof}
The identity \eqref{eq:modular-A} identifies the differential with the Stieltjes pullback. Apply \cref{cor:four}, \cref{lem:integration}, \cref{thm:rigidity}, and \cref{thm:family}.
\end{proof}

A form with a nontrivial character can instead have a fractional expansion $q^\nu f(q)$. Its quotient by $\theta t$ need not be rational in $t$. This is precisely the situation in the reducible modular constructions of \cite{SSintegrals,BSS}. Theorem~\ref{thm:kummer} treats the corresponding connection over the original annulus. The next application verifies its analytic hypotheses explicitly, so no convergence assumption is left implicit.

\subsection{The Hurwitz--Klein family}

Use the normalized invariant $J=j/1728$, so that its elliptic values are $0$ and $1$. The classical Fourier expansions and product formulas \cite{Zagier} are
\begin{equation}\label{eq:modular-series}
 \begin{split}
 \mathcal E(q)&=\prod_{k\ge1}(1-q^k),\qquad
 \eta(q)=q^{1/24}\mathcal E(q),\qquad \Delta=q\mathcal E^{24},\\
 E_2&=1-24\sum_{k\ge1}\sigma_1(k)q^k,\qquad
 E_6=1-504\sum_{k\ge1}\sigma_5(k)q^k,\\
 J&=\frac{E_4^3}{1728\Delta}.
 \end{split}
\end{equation}
All integral-power series in this display converge on $0<|q|<1$ over every $\mathbb Q_p$. The symbol $\eta$ may be interpreted on $q=z^{24}$; its logarithmic derivative belongs to the original ring. We use the classical identities \cite[equation~(23) and Proposition~15]{Zagier}
\begin{equation}\label{eq:ramanujan}
 \theta E_2=\frac{E_2^2-E_4}{12},\quad
 \theta E_4=\frac{E_2E_4-E_6}{3},\quad
 \theta E_6=\frac{E_2E_6-E_4^2}{2},\quad
 E_4^3-E_6^2=1728\Delta.
\end{equation}
Logarithmic differentiation of the product gives $\theta\eta/\eta=E_2/24$.

For $n\ge0$, let $Q_n\in\mathbb Q[x]$ be the monic polynomial of degree $n$ satisfying
\begin{equation}\label{eq:modular-jacobi}
 x(1-x)Q_n''+\left(\frac23-\frac76x\right)Q_n'
             +n\left(n+\frac16\right)Q_n=0.
\end{equation}
It is a scalar normalization of $P_n^{(-1/2,-1/3)}(2x-1)$, so its roots are simple and lie in $(0,1)$ when $n>0$ \cite[\S\S18.8,18.16(ii)]{DLMF}. Existence and uniqueness can equally be seen by solving the triangular coefficient equations, whose diagonal differences are $n(n+1/6)-k(k+1/6)$ for $0\le k<n$. For example,
\[
 Q_0=1,\qquad Q_1=x-\frac47,\qquad
 Q_2=x^2-\frac{20}{19}x+\frac{40}{247}.
\]
The differential $\eta^4Q_n(J)^{-2}\,dq/q$ is the classical weight-two construction of \cite[Theorems~7.2--7.3 and Corollary~7.4]{SSintegrals}; the Jacobi identification is \cite[Theorem~3.1]{BSS}. We derive the coefficient below in our normalization.

\begin{lemma}[Classical modular Stieltjes identity; {\cite[Theorem~7.3]{SSintegrals}}]\label[lemma]{lem:modular-identity}
Put $g_n=\eta^4/Q_n(J)^2$, interpreted as a Kummer developing function. Its coefficient is
\begin{equation}\label{eq:modular-F}
 F_{g_n}=-\frac{(12n+1)^2}{72}E_4.
\end{equation}
The weighted Stieltjes data on the $J$-line have logarithmic derivative
\begin{equation}\label{eq:modular-a}
 a(x)=-\frac{2}{3x}-\frac{1}{2(x-1)},\qquad
 V_n(x)=\frac{n(n+1/6)}{x(1-x)}.
\end{equation}
\end{lemma}

\begin{proof}
We verify the classical identity in the normalization of \eqref{eq:modular-series}, using \cite[Proposition~15]{Zagier}. The identities \eqref{eq:ramanujan} give
\[
 \theta J=-J\frac{E_6}{E_4},\qquad
 (\theta J)^2=J(J-1)E_4.
\]
If $A=\eta^4/(\theta J)$, logarithmic differentiation with respect to $J$ gives
\[
 \frac{1}{\theta J}\left(\frac{E_2}{6}
                   -\frac{\theta^2J}{\theta J}\right)
 =-\frac{2}{3J}-\frac{1}{2(J-1)}.
\]
Thus $A$ is a local branch of a constant multiple of $x^{-2/3}(x-1)^{-1/2}$, and \eqref{eq:modular-jacobi} is exactly the weighted Stieltjes equation with \eqref{eq:modular-a}. For $n=0$, the logarithmic derivative of $g_0=\eta^4$ is $E_2/6$, so the first Ramanujan identity gives $F_{g_0}=-E_4/72$. Proposition~\ref{prop:stieltjes}, Remark~\ref{rem:weighted}, and the chain rule \eqref{eq:chain} give
\[
 F_{g_n}-F_{g_0}
 =2(\theta J)^2V_n(J)=-2n(n+1/6)E_4,
\]
which is \eqref{eq:modular-F}.
\end{proof}

The identity is over $\mathbb Q$ at the level of formal expansions. Its local modular interpretation is classical; the next theorem describes its annular $p$-adic connection and Frobenius structures.

\begin{theorem}[Annular identification and least Frobenius iterate]\label[theorem]{thm:modular-family}
Let $p$ be any prime and $n\ge0$. Over $\mathcal R_{\mathbb Q_p}$, the differential module
\[
 \mathcal M_n:\quad
 \theta^2y-\frac{(12n+1)^2}{144}E_4(q)y=0
\]
is isomorphic to $\theta Z=\Lambda_0 Z$, where
\[
 \Lambda_0=\diag(-1/12,1/12).
\]
In particular, all $\mathcal M_n$ are mutually isomorphic as annular differential modules. For $d=p^s$, they admit Frobenius exactly when
\begin{equation}\label{eq:mod12}
 d\equiv1\quad\text{or}\quad d\equiv-1\pmod{12}.
\end{equation}
At primes $p\ge5$, the least exponent $s$ is one for $p\equiv1,11\pmod{12}$ and two for $p\equiv5,7\pmod{12}$. At $p=2,3$ no positive exponent $s$ exists.
\end{theorem}

\begin{proof}
We first verify the annular units. On $|q|<1$, the series $\mathcal E$, $E_4$, and $E_6$ have constant term one and integral coefficients, hence are units of norm one at every radius. Write
\[
 J=(1728q)^{-1}U(q),\qquad U\in1+q\mathbb Z_p[[q]].
\]
Thus $|J|_t=|1728|_p^{-1}t^{-1}$. In a finite splitting field for $Q_n$, let $\xi$ be one of its roots and put $L=|1728|_p^{-1}$. If $|\xi|\le L$, then $J-\xi=J(1-\xi/J)$ is a unit on every annulus in the punctured disc. If $|\xi|>L$, then on $t>L/|\xi|$ the factorization $J-\xi=-\xi(1-J/\xi)$ proves invertibility. Choosing one inner radius for the finitely many roots proves that $Q_n(J)$ and its inverse belong to a Robba ring. The inverse descends to $\mathbb Q_p$ by uniqueness, or directly by its rational expression.

Consequently,
\[
 w_n=\frac{\mathcal E^2}{Q_n(J)}\in\mathcal R_{\mathbb Q_p}^{\times},
 \qquad g_n=q^{1/6}w_n^2.
\]
Set $f_n=w_n^2$, $H_n=(\theta+1/6)^{-1}f_n$, and
\begin{equation}\label{eq:modular-P}
 a_n=-\frac12\left(\frac16+\frac{\theta f_n}{f_n}\right),\qquad
 \mathcal P_n=w_n^{-1}
 \begin{pmatrix}1&H_n\\a_n&f_n+a_nH_n\end{pmatrix}.
\end{equation}
\cref{lem:twisted} gives $H_n$ on the same open annulus, and $\det\mathcal P_n=1$. \cref{cor:kummer-square} gives the normal form $\Lambda_0$ at every prime.

For this diagonal system, a nonzero Frobenius entry has exponent $(d-1)/12$, $-(d-1)/12$, $-(d+1)/12$, or $(d+1)/12$. An invertible diagonal or antidiagonal pair exists precisely under \eqref{eq:mod12}. Explicit choices are
\begin{equation}\label{eq:modular-C}
 C_d=\begin{cases}
 \diag(q^{(d-1)/12},q^{-(d-1)/12}),&d\equiv1\pmod{12},\\[1mm]
 \begin{pmatrix}0&q^{-(d+1)/12}\\q^{(d+1)/12}&0\end{pmatrix},&d\equiv-1\pmod{12}.
 \end{cases}
\end{equation}
The original companion-basis Frobenius matrix is
\[
 \Phi_{n,d}=\mathcal P_n C_d\varphi_d(\mathcal P_n)^{-1}.
\]
Finally, every prime $p\ge5$ is $1,5,7$, or $11$ modulo $12$, and each of these classes has square one. Powers of $2$ or $3$ are never $\pm1$ modulo $12$. This proves the assertions about the least exponent $s$.
\end{proof}

The coefficients in \eqref{eq:modular-F} vary with $n$, while \eqref{eq:modular-P} identifies their connections on an outer annular germ. At $p\equiv11\pmod{12}$, Frobenius interchanges the two constituents; at $p\equiv5,7\pmod{12}$, the pair closes only after squaring the Frobenius substitution.

\subsection{The four modular congruence classes}
The other three reducible modular families have the same annular treatment. Their classical developing differentials and Jacobi description are given in \cite[\S\S2--3]{BSS}. The following consequence identifies exactly how much of their parameter survives in the annular differential module.

\begin{corollary}[Two annular types for the reducible modular family]\label[corollary]{cor:all-modular}
For every positive integer $\ell$ coprime to $6$, let
\[
 \mathcal N_\ell:\quad \theta^2y-\frac{\ell^2}{144}E_4y=0.
\]
Write $\ell=12n+e$, where $n\geq 0$ and $e\in\{1,5,7,11\}$. Over $\mathcal R_{\mathbb Q_p}$, at every prime, this module is isomorphic to
\[
 \theta Z=\diag(-e/12,e/12)Z.
\]
Two such modules $\mathcal N_{\ell_1}$ and $\mathcal N_{\ell_2}$ are isomorphic exactly when $\ell_1\equiv\ell_2$ or $-\ell_2\pmod{12}$. Thus the full family has exactly two annular differential-module types. Every member has the Frobenius existence criterion and least exponents $s$ of Theorem~\ref{thm:modular-family}.
\end{corollary}

\begin{proof}
The Jacobi pole polynomial and developing differential used below arise
from the reducible modular construction of
\cite[Theorem~3.1, Corollary~4.4, and Theorem~6.1]{BSS}. We verify the Schwarzian coefficient directly in our normalization
before applying the annular classification. Choose $\epsilon,\delta\in\{0,2\}$ such that
$e=1+2\epsilon+3\delta$, and set
\[
\alpha=\frac{\epsilon-1}{3},
\qquad
\beta=\frac{\delta-1}{2}.
\]
They are rational and greater than $-1$. Let $Q_{n,e}$ be the
monic normalization of the shifted Jacobi polynomial
\[
P_n^{(\beta,\alpha)}(2x-1).
\]
The finite representation of the Jacobi polynomials in
\cite[Equation~(18.5.7)]{DLMF} shows that
$Q_{n,e}\in\mathbb Q[x]$ and has degree $n$. The Jacobi differential
equation \cite[Table~18.8.1]{DLMF} gives
\begin{equation}\label{eq:four-jacobi}
 x(1-x)Q_{n,e}''
 +\left(\frac{2+\epsilon}{3}
        -\left(1+\frac e6\right)x\right)Q_{n,e}'
 +n\left(n+\frac e6\right)Q_{n,e}=0.
\end{equation}
This monic solution is unique: the difference of two such solutions,
if nonzero of degree $k<n$, would give
$k\left(k+\frac e6\right)
=n\left(n+\frac e6\right)$
upon comparison of leading coefficients, which is impossible.
Since the polynomial and the differential equation are defined over
$\mathbb Q$, they remain valid after scalar extension to $\mathbb Q_p$. Put
\[
 g_{n,e}=\frac{\eta^{4+8\epsilon+12\delta}}
 {E_4^\epsilon E_6^\delta Q_{n,e}(J)^2},
\]
interpreted as a Kummer developing function. We verify its coefficient before applying the annular calculation. For $n=0$, its logarithmic derivative is
\[
 u_{0,e}=\frac{E_2}{6}+\frac\epsilon3 R+\frac\delta2 T,
 \qquad R=\frac{E_6}{E_4},\quad T=\frac{E_4^2}{E_6}.
\]
Ramanujan's identities give
\[
 \theta R-\frac{E_2R}{6}=-\frac{E_4}{2}+\frac{R^2}{3},\qquad
 \theta T-\frac{E_2T}{6}=-\frac{2E_4}{3}+\frac{T^2}{2},\qquad RT=E_4.
\]
Since $\epsilon^2=2\epsilon$ and $\delta^2=2\delta$, expansion of $\theta u_{0,e}-u_{0,e}^2/2$ cancels the $R^2$ and $T^2$ terms and gives $-e^2E_4/72$. Writing locally $g_{0,e}=A_e(J)\theta J$
and setting
\[
a_e(x)=\frac{A_e'(x)}{A_e(x)},
\]
we obtain the weighted logarithmic derivative
\[
a_e(x)
=-\frac{2+\epsilon}{3x}
-\frac{1+\delta}{2(x-1)}.
\]
Indeed, relative to the case $e=1$, the identities
$\Delta/E_4^3=(1728J)^{-1}$ and
$\Delta/E_6^2=(1728(J-1))^{-1}$ add the factors
$x^{-\epsilon/3}(x-1)^{-\delta/2}$. Taking logarithmic derivatives and using
$a_1(x)=-\frac{2}{3x}-\frac{1}{2(x-1)}$
from \eqref{eq:modular-a}, we obtain
\[
\begin{aligned}
a_e(x)
=a_1(x)-\frac{\epsilon}{3x}
             -\frac{\delta}{2(x-1)}
=-\frac{2+\epsilon}{3x}
  -\frac{1+\delta}{2(x-1)}.
\end{aligned}
\]
Equation~\eqref{eq:four-jacobi} gives
$V_{n,e}=n(n+e/6)/[x(1-x)]$.
Proposition~\ref{prop:stieltjes}, Remark~\ref{rem:weighted}, and the
chain rule~\eqref{eq:chain} therefore give
\[
\begin{aligned}
F_{g_{n,e}}-F_{g_{0,e}}
 =2(\theta J)^2V_{n,e}(J)
 =-2n\left(n+\frac e6\right)E_4.
\end{aligned}
\]
Consequently,
\[
F_{g_{n,e}}
=-\frac{e^2}{72}E_4
 -2n\left(n+\frac e6\right)E_4
=-\frac{(12n+e)^2}{72}E_4.
\]
Hence the Schwarzian module determined by $g_{n,e}$ is
$\mathcal N_\ell$, where $\ell=12n+e$.
The unit argument in Theorem~\ref{thm:modular-family} applies to every polynomial $Q_{n,e}$. Thus
\[
 g_{n,e}=q^{e/6}w_{n,e}^2,\qquad
 w_{n,e}=\frac{\mathcal E^{2+4\epsilon+6\delta}}
 {E_4^{\epsilon/2}E_6^{\delta/2}Q_{n,e}(J)}
 \in\mathcal R_{\mathbb Q_p}^{\times}.
\]
Since $e/6\notin\mathbb Z$ and $w_{n,e}^2$ is an explicit square,
\cref{cor:kummer-square} applies at every prime. Taking
$\nu=e/6$ and $f=w_{n,e}^2$, we obtain the gauge in
\eqref{eq:modular-P} with $w_n$ replaced by $w_{n,e}$ and $1/6$
replaced by $e/6$. It transforms the module into
\[
\theta Z=\diag(-e/12,e/12)Z.
\]
The resulting determinant-one gauge gives the claimed diagonal system. For $\ell_i=12n_i+e_i$, set
\[
(\lambda_1,\lambda_2)=(-e_1/12,e_1/12),
\qquad
(\mu_1,\mu_2)=(-e_2/12,e_2/12).
\]
A gauge between the corresponding diagonal systems has entries
satisfying $\theta C_{ij}=(\lambda_i-\mu_j)C_{ij}$. By \cref{lem:integration}, either $C_{ij}=0$, or
$\lambda_i-\mu_j\in\mathbb Z$ and
\[
C_{ij}=c_{ij}q^{\lambda_i-\mu_j},
\qquad c_{ij}\in\mathbb Q_p.
\]
An invertible gauge therefore exists precisely when the two exponent
pairs agree, up to permutation, modulo $\mathbb Z$. Hence
$\mathcal N_{\ell_1}\simeq\mathcal N_{\ell_2}$ if and only if
\[
\ell_1\equiv\ell_2
\quad\text{or}\quad
\ell_1\equiv-\ell_2\pmod{12}.
\]
The four admissible congruence classes consequently form the two classes
$\{\pm1\}$ and $\{\pm5\}$ modulo $12$.

Finally, Frobenius multiplication sends the unordered exponent pair
$\{\pm e/12\}$ to $\{\pm de/12\}$ in $\mathbb Q/\mathbb Z$. It
preserves this pair precisely when
$de\equiv \pm e\pmod{12}$.
Since $\gcd(e,12)=1$, this condition is equivalent to
$d\equiv\pm1\pmod{12}$,
which is exactly \eqref{eq:mod12}. When this condition holds, an
explicit Frobenius matrix in the diagonal basis is obtained from
$C_d$ in \eqref{eq:modular-C} by replacing each power $q^\alpha$
with $q^{e\alpha}$. Write $C_{d,e}$ for this matrix and $\mathcal P_{n,e}$ for the
preceding determinant-one gauge. The Frobenius matrix in the
original companion basis is
\[
\Phi_{n,e,d}
=\mathcal P_{n,e}C_{d,e}
 \varphi_d(\mathcal P_{n,e})^{-1}.
\]
\end{proof}
For $p\equiv5,7\pmod{12}$, pullback by $\varphi_p$ exchanges the two annular types in Corollary~\ref{cor:all-modular}, because multiplication by $p$ exchanges the exponent pairs $\{\pm1/12\}$ and $\{\pm5/12\}$ modulo $\mathbb Z$. This explains why two iterations are necessary and sufficient.

The arithmetic Frobenius action on the roots of the Jacobi pole polynomial is studied in \cite[Proposition~4.9]{BSS}. The structures classified here act on the annular differential module under $q\mapsto q^d$ and are constructed over $\mathbb Q_p$ without adjoining those roots. A geometric cohomological interpretation would require a comparison with its specified Frobenius map.


\begin{thebibliography}{99}
\bibitem{BCG}
M. Bertola, E. Chavez-Heredia, and T. Grava,
\href{https://doi.org/10.1093/imrn/rnae037}
{The Stieltjes--Fekete problem and degenerate orthogonal polynomials},
\emph{International Mathematics Research Notices} \textbf{2024}, no.~11,
9114--9141.

\bibitem{BSS}
K.~Besrour, H.~Saber, and A.~Sebbar,
\emph{Reducible modular differential equations, Jacobi pole divisors, and supersingular lifts},
preprint, \href{https://arxiv.org/abs/2508.10788v2}{arXiv:2508.10788v2}, 2026.

\bibitem{BGR}
S.~Bosch, U.~G\"untzer, and R.~Remmert,
\emph{Non-Archimedean Analysis: A Systematic Approach to Rigid Analytic Geometry},
Grundlehren der mathematischen Wissenschaften, vol.~261,
Springer, Berlin, 1984.

\bibitem{KedOverview}
K.~S. Kedlaya,
\href{https://doi.org/10.1142/S179304210500008X}{\emph{Local monodromy of $p$-adic differential equations: an overview}},
Int. J. Number Theory \textbf{1} (2005), no.~1, 109--154.

\bibitem{Kedlaya}
K.~S. Kedlaya,
\href{https://doi.org/10.1017/CBO9780511750922}{\emph{$p$-adic Differential Equations}},
Cambridge Studies in Advanced Mathematics, vol.~125,
Cambridge University Press, Cambridge, 2010.

\bibitem{KedFamilies}
K.~S. Kedlaya,
\emph{Monodromy representations of $p$-adic differential equations in families},
in \emph{Recent Advances in $p$-adic Hodge Theory}
(B.~Bhatt and M.~Olsson, eds.), Simons Symposia,
Springer, Cham, to appear;
\href{https://arxiv.org/abs/2209.00593v3}{arXiv:2209.00593v3}.

\bibitem{Korotkin}
D.~Korotkin,
\href{https://doi.org/10.1016/j.nuclphysb.2017.12.019}{\emph{Stieltjes--Bethe equations in higher genus and branched coverings with even ramifications}},
Nucl. Phys. B \textbf{927} (2018), 294--318.

\bibitem{MV}
E.~Mukhin and A.~Varchenko,
\href{https://doi.org/10.1142/S0219199704001288}{\emph{Critical points of master functions and flag varieties}},
Commun. Contemp. Math. \textbf{6} (2004), no.~1, 111--163.

\bibitem{DLMF}
F.~W.~J. Olver, A.~B. Olde Daalhuis, D.~W. Lozier, B.~I. Schneider,
R.~F. Boisvert, C.~W. Clark, B.~R. Miller, B.~V. Saunders, H.~S. Cohl,
and M.~A. McClain (eds.),
\href{https://dlmf.nist.gov/}{\emph{NIST Digital Library of Mathematical Functions}},
version 1.2.7, June 15, 2026.

\bibitem{SSintegrals}
H.~Saber and A.~Sebbar,
\href{https://doi.org/10.1007/s11139-020-00348-w}{\emph{Automorphic Schwarzian equations and integrals of weight $2$ forms}},
Ramanujan J. \textbf{57} (2022), no.~2, 551--568.

\bibitem{Sasaki}
R.~Sasaki,
\href{https://doi.org/10.1063/1.4918707}{\emph{Perturbations around the zeros of classical orthogonal polynomials}},
J. Math. Phys. \textbf{56} (2015), no.~4, 042106.

\bibitem{Scherbak}
I.~Scherbak,
\href{https://doi.org/10.1112/S0024610704005733}{\emph{Intersections of Schubert varieties and critical points of the generating function}},
J. London Math. Soc. (2) \textbf{70} (2004), no.~3, 625--642.

\bibitem{Stieltjes}
T.~J. Stieltjes,
\href{https://doi.org/10.1007/BF02400421}{\emph{Sur certains polyn\^omes qui v\'erifient une \'equation diff\'erentielle lin\'eaire du second ordre et sur la th\'eorie des fonctions de Lam\'e}},
Acta Math. \textbf{6} (1885), 321--326.

\bibitem{Zagier}
D.~Zagier,
\href{https://doi.org/10.1007/978-3-540-74119-0_1}{\emph{Elliptic modular forms and their applications}},
in \emph{The 1-2-3 of Modular Forms} (K.~Ranestad, ed.), Universitext,
Springer, Berlin, 2008, pp.~1--103.

\end{thebibliography}
\end{document}